\documentclass{article}
\usepackage[a4paper,top=3cm,bottom=3cm,left=2.1cm,right=2.1cm,marginparwidth=1.75cm]{geometry}

\usepackage{graphicx} 

\usepackage{bm}
\usepackage{amsmath, amsfonts, amssymb, amsthm}
\usepackage{mathrsfs, dsfont}
\usepackage[dvipsnames]{xcolor}
\usepackage{graphicx}
\usepackage{verbatim}
\usepackage{float}
\usepackage{comment}
\usepackage{ulem}
\usepackage{subfig}
\usepackage{bigints}

\usepackage[colorlinks=true, allcolors=blue]{hyperref}
\usepackage{xparse} 
\usepackage{hyperref}
\usepackage{tasks}
\settasks{style=itemize}
\usepackage{parskip}
\usepackage{todonotes}

\usepackage[round]{natbib}

\newtheorem{theorem}{Theorem}[section]
\newtheorem{lemma}[theorem]{Lemma}
\newtheorem{definition}[theorem]{Definition}

\newtheorem{proposition}[theorem]{Proposition}

\newtheorem{corollary}[theorem]{Corollary}

\theoremstyle{remark}
\newtheorem{remark}[theorem]{Remark}
\newtheorem{example}[theorem]{Example}

\numberwithin{equation}{section}

\newenvironment{sqremark}{\begin{remark}}{\hfill \tiny $\blacksquare$ \end{remark}}

\usepackage{pifont}

\DeclareMathOperator*{\esssup}{ess\,sup}

\newcommand{\R}{\mathbb{R}}
\newcommand{\N}{\mathbb{N}}

\newcommand{\Q}{\mathbb{Q}}

\newcommand{\E}{\mathbb{E}}
\def\P{\mathbb{P}} 
\newcommand{\F}{\mathcal{F}}

\newcommand{\indic}[1]{\mathds{1}_{\left\{ #1 \right\}}}

\newcommand{\eTA}[1][d]{T((\R^{#1}))}

\newcommand{\emptyword}{{\color{NavyBlue}{\mathbf{\varnothing}}}}
\newcommand{\word}[1]{{\color{NavyBlue}{\mathbf{#1}}}}
\newcommand{\proj}[1]{|_{\word{#1}}}

\newcommand{\bp}{\bm{p}}
\newcommand{\bq}{\bm{q}}

\newcommand{\wv}{\word{v}}

\newcommand{\bpsi}{\bm{\psi}}
\newcommand{\bu}{\bm{u}}

\newcommand{\bchi}{\bm{\chi}}

\newcommand{\bell}{\bm{\ell}}
\newcommand{\br}{\bm{r}}

\newcommand{\shuprod}{\mathrel{\sqcup \mkern -3mu \sqcup}}
\newcommand{\shupow}[1]{^{\shuprod #1}}
\newcommand{\shuexp}[1]{\exp^{\shuprod}\left({#1}\right)}

\newcommand{\conpow}[1]{^{\otimes #1}}

\NewDocumentCommand{\sighat}{O{t} O{W}}{\widehat{\mathbb{#2}}_{#1}}
\NewDocumentCommand{\sigtilde}{O{t} O{W}}{\widetilde{\mathbb{#2}}_{#1}}
\NewDocumentCommand{\sig}{O{t} O{W}}{\sighat[#1][#2]}
\newcommand{\sigW}[1][t]{\widehat{\mathbb{W}}_{#1}}

\newcommand{\sigX}[1][t]{{\mathbb{X}}_{#1}}
\newcommand{\sigY}[1][t]{{\mathbb{Y}}_{#1}}
\newcommand{\sigXhat}[1][t]{{\widehat{\mathbb{X}}}_{#1}}
\newcommand{\sigYhat}[1][t]{{\widehat{\mathbb{Y}}}_{#1}}
\newcommand{\sigXahat}[1][t]{{\widehat{\mathbb{X}}^{\alpha}}_{#1}}
\newcommand{\sigXastarhat}[1][t]{{\widehat{\mathbb{X}}^{\alpha^*}}_{#1}}

\newcommand{\bracket}[2]{ \langle #1,\, #2  \rangle}
\NewDocumentCommand{\bracketsig}{O{t} O{W} m}{\bracket{#3}{\sig[#1][#2]}}

\NewDocumentCommand{\bracketsigtrunc}{O{M} O{t} O{W} m}{\bracket{#4}{\sig[#2][#3]^{\leq #1}}}
\NewDocumentCommand{\bracketsigtilde}{O{t} O{W} m}{\bracket{#3}{\sigtilde[#1][#2]}}

\usepackage{mathtools}
\usepackage{shuffle}
\usepackage{bbm}
\usepackage{authblk}
\mathtoolsset{showonlyrefs}
\newcommand{\Dim}{\color{PineGreen}}

\usepackage{enumitem}
\usepackage{stmaryrd}

\title{
Stochastic control with signatures via Riccati equations on the tensor algebra}

\author[1]{Eduardo Abi Jaber\thanks{eduardo.abi-jaber@polytechnique.edu.  EAJ is grateful for the financial support from the Chaires FiME-FDD and  Financial Risks at Ecole Polytechnique.}}
\author[1]{Elie Attal\thanks{elie.attal@polytechnique.edu.}}
\author[1, 2]{Dimitri Sotnikov\thanks{dmitrii.sotnikov@polytechnique.edu. DS is grateful for the financial support provided by Engie Global Markets.}}
\affil[1]{Ecole Polytechnique, CMAP}
\affil[2]{Engie Global Markets}

\begin{document}

\maketitle

\begin{abstract}
 We solve in semi-explicit form a class of non-Markovian stochastic optimal control problems with path-dependent rewards, using path signatures.~We reformulate the control problem as the computation of Laplace transforms of signature functionals thanks to the Boué--Dupuis representation.~Exploiting recent signature representations of such transforms on tensor algebras, we determine the value process and the optimal control through an infinite-dimensional system of Riccati equations on the extended tensor algebra. We establish an explicit feedback representation of the optimal control and the value process as an infinite linear combination  of the time-extended signature of the controlled process, with time-dependent coefficients. The expansions being intrinsically local, we propose a dynamic recentering algorithm to ensure a global representation over the entire time horizon. We illustrate the approach on genuinely path-dependent, non-linear examples that go beyond the tractable linear-quadratic setting, including the tracking of linear functionals of the signature and signature lifts of Volterra control problems.
 \end{abstract}




\section{Introduction}

Consider the controlled Brownian motion
\begin{align}\label{eq:introWalpha}
dX_t^\alpha=\alpha_t\,dt+dW_t,    
\end{align}
and the path-dependent stochastic control problem
\begin{align}\label{eq:introobjective}
\sup_{\alpha}
\mathbb{E}\Big[
- \frac 1 2  \int_0^T \alpha_s^2\,ds
+
F\big(X^\alpha\big)
\Big],
\end{align}
where $F(X^\alpha) = F((X_s^\alpha)_{s \in [0, T]})$ is a functional of the entire trajectory.

The problem is well understood with explicit solutions for certain special choices of $F$. For instance, 
$$
F(X)=X_T^2,
\qquad\text{or}\qquad
F(X)
=
\int_0^T X_s^2\,ds,
$$
lead to classical linear--quadratic (LQ) control problems. In these Markovian settings, the value function is quadratic in the state variable, and the optimal control admits a linear feedback form with explicit time-dependent coefficients, which themselves solve a finite-dimensional Riccati equation.  Dynamic programming yields a Hamilton--Jacobi--Bellman equation that can be solved explicitly either by making a quadratic ansatz for the value function, which leads to the associated Riccati equation, or by applying a Cole--Hopf transformation to linearize the partial differential equation, 
 see \citep*{yong1999stochastic}.

A similar tractability persists in certain non-Markovian situations. For example, if
$$
F(X)
=
\Big(
\int_0^T K(T-s)\, dX_s
\Big)^2,
$$
for a locally square-integrable kernel $K$, one obtains a linear--quadratic Volterra control problem. Such problems can be analyzed either through variational arguments leading to linear stochastic Fredholm equations that can be solved explicitly  \citep*{abijaber2023equilibrium}, or by introducing suitable infinite-dimensional state variables that restore the Markov property and yield Riccati equations in an extended state space \citep*{abi2021linear}.

A probabilistic viewpoint that unifies these examples is provided by the \cite*{boue1998variational}  variational representation:
\begin{align}\label{eq:introBD}
\sup_\alpha
\mathbb{E}
\Big[
- \frac 1 2  \int_0^T \alpha_s^2\,ds
+
F(X^\alpha)
\Big]
=
\log
\mathbb{E}
\Big[
\exp({F(W)})
\Big],  
\end{align}
This identity reveals the source of tractability in the previous examples. Indeed, both the Markovian and Volterra LQ cases reduce to computing Laplace transforms of quadratic Gaussian functionals. Such quantities are explicit since they correspond to generalized chi-squared distributions and can be characterized through \cite*{fredholm1903classe}  determinants and resolvents. 

The situation becomes much less clear for general path-dependent and non-quadratic functionals $F$. While the Boué–Dupuis representation \eqref{eq:introBD} still provides an exact characterization of the value of the control problem, allowing for its approximation via Monte Carlo methods through the right-hand side, it does not by itself provide a tractable characterization of optimal controls. Under suitable regularity assumptions on $F$,        an optimal control can be identified with the \cite*{follmer1985entropy}  drift associated with the Gibbs measure
$$
d\mathbb Q^*=\frac{\exp({F(W)})}{\mathbb E[\exp({F(W)})]}d\mathbb P,
$$
namely the unique adapted process $\alpha^*$ such that
$
W_t-\int_0^t \alpha_s^* \,ds
$
is a Brownian motion under $\mathbb Q^*$. When $F$ is Malliavin differentiable, this drift admits the explicit Clark--Ocone representation
\begin{equation}\label{eq:clark_ocone_repr}
    \alpha_t^*=
\frac{\mathbb E\left[\exp({F(W)})\mathbf{D}_tF(W)\mid \mathcal F_t\right]}
{\mathbb E\left[\exp({F(W)})\mid \mathcal F_t\right]},
\end{equation}
see \cite*{Lehec2013}.
Although explicit, this representation generally lacks exploitable structure for analysis or numerical computation.

In such non-Markovian settings, the main challenge is therefore to identify suitable state variables that restore tractability while preserving the path-dependent nature of the objective. Ideally, one seeks a representation that linearizes the dependence on the path, in much the same way that finite-dimensional state augmentation restores Markovianity in classical control problems.

\paragraph{Our contribution.}   The present work develops a tractable framework to tackle a  broad class of non-Markovian  functionals $F$ based on path signatures. Signatures, which are  sequences of iterated integrals introduced by \cite*{chen1957integration},  play the role of polynomials on path space and form a universal family  for continuous paths. By Stone--Weierstrass approximation results, a large class of path functionals can be approximated by linear functionals of the signature; see \cite*{levin2013learning}. Motivated by this universality, we focus on the case where the terminal reward functional $F$ in \eqref{eq:introobjective} is given by a finite linear combination of time-extended signatures of the controlled process $X^{\alpha}$. Combining this representation with the conditional Boué--Dupuis formula transforms the original control problem into the computation of conditional Laplace transforms of linear signature functionals. Exploiting the algebraic structure of signatures together with recent  signature expansions for log-Laplace transforms on tensor algebras derived by \cite*{riccatifourierlaplace}, we show that the resulting control problem can be solved through an infinite-dimensional system of Riccati equations on the extended tensor algebra.

More precisely, our main Theorem~\ref{thm:oc_sig_repr} shows that the optimal control admits a feedback representation as a local signature expansion, whose coefficients solve a Riccati equation on the extended tensor algebra; see \eqref{eq:value_control_local}. In particular, the optimal control is linear in the signature components of the controlled process, revealing a generalized affine structure reminiscent of the classical LQ setting. Moreover, when the functional reduces to a Markovian Linear–Quadratic form, our representation recovers the classical LQ formulas; see Remark~\ref{Ex:LQ}. More generally, in the Markovian setting beyond the LQ case, our formulas correspond precisely to the analytic expansion in the $x$-variable of the solution to the associated Hamilton–Jacobi–Bellman equation; see Remark~\ref{rmq:hjb}. 
 We emphasize, however, that the functional considered here is neither Linear–Quadratic nor Markovian in the controlled variable.   Unlike in the classical LQ framework, our representation is intrinsically local. Recovering a global representation requires {dynamic recentering of the expansion}, which leads to  randomized Riccati equations with path-dependent terminal conditions; see \eqref{eq:Riccati_recentered}.

We illustrate the flexibility of our framework in Section~\ref{section:application} through genuinely path-dependent and nonlinear examples beyond the Linear–Quadratic case, including {tracking linear functionals of the signature}, and show that it naturally encompasses signature lifts of a broad class of Volterra control problems.

Our results provide a theoretical framework for stochastic control in non-Markovian settings beyond the Linear–Quadratic paradigm. They yield an explicit feedback representation of the optimal control in terms of the time-extended signature of the controlled process, thereby establishing a direct connection between non-Markovian stochastic control and Riccati equations on tensor algebras. In contrast to existing signature approaches,  which parametrize controls via signatures and optimize over restricted classes,  we stress  that our method solves the original stochastic control problem over the full class of admissible progressively measurable controls.  
We also expect the underlying signature/Riccati structure to extend to more general control problems and dynamics, although the corresponding questions of existence and convergence for the Riccati system remain highly nontrivial, see for instance~\cite*{abijaber2025signature}.

\paragraph{Related literature.} Our work connects to the growing literature on the use of signatures in stochastic control. Broadly speaking, existing approaches can be divided into three main directions.

The first direction uses signatures as feature maps, or nonlinear transformations thereof, to encode path dependence, as developed for instance in \cite*{kidger2019deep}. In this approach, signatures provide finite-dimensional summaries of historical trajectories, which are then used as inputs to learning architectures or approximation schemes. For instance, \cite*{bayer2024pricing} combines deep-signature and signature-kernel methods to solve non-Markovian optimal stopping problems, while \cite*{abi2025signature} uses signatures as inputs to shallow feedforward neural networks that parametrize the optimal controls for hedging under non-Markovian stochastic volatility, showing that they can outperform recurrent architectures at a  lower training cost.

The second direction uses signatures to parameterize admissible controls, typically through linear functionals of the signature, with coefficients chosen as optimization variables. This approach is motivated by universal approximation results, relying on Stone--Weierstrass arguments. Early examples of this viewpoint appear in \cite*{lyons2020non} in the context of pricing and hedging. More recently, \cite*{bank2024stochastic} introduced general classes of  signature controls, proved their density in the set of admissible progressively measurable controls, and proposed associated numerical methods for computing the optimal control {whereas \cite*{aqsha2026solving} presents a numerical approach based on signature expansions to solve LQ stochastic control problems}. Financial applications and variations of this approach appear in \cite*{cuchiero2025signature,futter2023signature,gennaro2026signature,  kalsi2020optimal}.

A third distinct direction, initiated by \cite*{abijaber2025signature}, lifts the control problem onto signature state variables and establishes, for the first time, a connection  between stochastic optimal control and Riccati equations on the extended tensor algebra. Under the assumption of existence of solutions to the associated Riccati system and convergence of the corresponding signature expansions, a verification theorem is derived, yielding the optimal control in feedback form as a possibly infinite linear combination of the time-extended signatures of the controlled variables, with time-dependent coefficients characterized by an infinite-dimensional Riccati equation. This framework was validated numerically in  non-Markovian settings, illustrating its effectiveness in path-dependent control problems. {However, several fundamental issues
remained unresolved. First, the verification theorem assumes global
convergence of the signature expansion, an assumption that cannot hold
in general: as shown by \cite*{riccatifourierlaplace}, the log-Laplace
expansion has an intrinsically finite radius of convergence outside the
linear-quadratic framework. Second, no class of admissible functionals
was identified for which the Riccati equation provably admits a solution. The present work addresses both issues through a different approach. Rather than pursuing the verification strategy of \cite*{abijaber2025signature}, we exploit the Bou\'e--Dupuis variational representation \eqref{eq:introBD}, which reduces the control problem to the computation of the conditional log-Laplace transforms of linear signature functionals. This reduction allows us to leverage the class $\mathcal B$ and the quantitative bounds from \cite*{riccatifourierlaplace}
 directly, yielding a rigorous framework. Moreover, we introduce a dynamic recentering procedure that overcomes the intrinsic locality of the expansion, producing a theoretically complete and implementable approach.} 

Our work  fully derives {an optimal closed-loop control}  in terms of the signature of the controlled process. In contrast with signature-parametrization approaches, such as \cite*{lyons2020non,bank2024stochastic}, which optimize over a prescribed subclass of controls with time-independent coefficients, our method solves the original stochastic control problem over the full class of admissible progressively measurable controls. The resulting optimal control is represented through genuinely time-dependent coefficients, yielding a richer class of controls and avoiding the analyticity  constraints (in the time variable)  implicit in time-independent parametrizations.

More broadly, related infinite-dimensional Riccati systems have also appeared in uncontrolled problems in mathematical finance, notably for computing characteristic functions of signature SDEs \citep*{cuchiero2025signaturesdesaffinepolynomial} and  signature volatility models \citep*{abi2025signature,abijaber2024fourier}. A general existence and uniqueness theory for such systems {in the genuinely path-dependent case}, together with convergence of the associated power expansions, is delicate and has only recently been established for time-augmented  signatures {of the one-dimensional Brownian motion} by \cite*{riccatifourierlaplace}. Earlier partial results on the Riccati side can be found in \cite*{abijaber2024fourier,cuchiero2025signaturesdesaffinepolynomial}, while results on infinite signature expansions for Markovian and non-Markovian systems appear in \cite*{abi2024path,arous1989flots}.

{\paragraph{Outline.} Section~\ref{section:sig_prelim} provides necessary preliminaries on path signatures. In Section~\ref{section:riccati_prelim}, we define our class of admissible reward functionals, and recall existence results for the Riccati equation on the tensor algebra, as well as conditions for convergence of the log-Laplace transform expansion. Section~\ref{section:oc} presents our main result, Theorem \ref{thm:oc_sig_repr}, in which we explicitly construct feedback maps in the signature of the controlled process for the value process and the optimal control. Section~\ref{section:application} treats the applications and validates the approach numerically, providing illustrations of the signature control with dynamic recentering.
}

\section{Preliminaries on signatures}\label{section:sig_prelim}

This section introduces the main objects and notations used throughout the
paper, we refer to \cite*{riccatifourierlaplace} for more details.

\paragraph{Time-augmented signatures and extended tensor algebra.}
Let $(X_t)_{t \geq 0}$ be a continuous, real-valued semimartingale. We define its
\textit{time-augmented signature} $\sigXhat[s,t]$ over $[s,t]$ as the path
signature of the time-augmented process $\widehat{X}_t := (t, X_t)$, that is,
the collection of all iterated Stratonovich integrals
$$
\sigXhat[s,t]^{\word{i_1 \ldots i_n}}
  := \int_{s \leq t_1 \leq \ldots \leq t_n \leq t}
     \circ d\widehat{X}^{i_1}_{t_1} \ldots \circ d\widehat{X}^{i_n}_{t_n}\,,
$$
for $n \geq 1$ and $i_1, \ldots, i_n \in \{0, 1\}$. These integrals are
indexed by words $\wv$ over the alphabet $A := \{\word{0}, \word{1}\}$,
where $\word{0}$ and $\word{1}$ correspond to integration against $dt$ and
$\circ dX_t$, respectively. We denote by $V$ the set of all finite words
over $A$, including the empty word $\emptyword$, and set
$\sigXhat[s,t]^{\emptyword} := 1$ by convention. We also use the notation $\sigXhat[t] := \sigXhat[0,t]\,$.

The time-augmented signature extends the notion of monomials from classical
analysis to path space: in particular,
$\sigXhat[s,t]^{\word{1 \ldots 1}} = (X_t - X_s)^n / n!$ when the letter
$\word{1}$ appears $n$ times.
Accordingly, the signature inherits an analogous universality: any
sufficiently regular functional $F((X_r - X_s)_{r \in [s,t]})$ can be
approximated by a linear combination of time-augmented signature components.

Identifying the letters $\word{0}$ and $\word{1}$ with the standard basis
vectors $(1,0)$ and $(0,1)$ of $\R^2$, the time-augmented signature
naturally lives in the \textit{extended tensor algebra}
$$
\eTA[2] := \prod_{n=0}^{\infty} (\R^2)^{\otimes n}
  = \bigl\{\bp = (\bp_0, \bp_1, \ldots, \bp_n, \ldots)
    \colon\ \bp_n \in (\R^2)^{\otimes n}\,,\; n \geq 0\bigr\}\,,
$$
where $(\R^2)^{\otimes 0} = \R$ corresponds to the empty-word component.
For each $n \geq 0$, the elementary tensors
$\word{i_1} \otimes \ldots \otimes \word{i_n}$, written as concatenated
words $\word{i_1 \ldots i_n}$, form a natural basis of
$(\R^2)^{\otimes n}$. Any element $\bp \in \eTA[2]$ can thus be formally
expanded as $\bp = \sum_{\wv \in V} \bp^{\wv}\, \wv$, where the
coefficients $\bp^{\wv}$ are real-valued.

We denote by $|\wv|_{\word{i}}$ the number of occurrences of the letter
$\word{i} \in \{\word{0}, \word{1}\}$ in the word $\wv$, and by
$|\wv| := |\wv|_{\word{0}} + |\wv|_{\word{1}}$ its length. For an element
$\bp \in \eTA[2]$, we define its degree and partial
degrees by
$\deg(\bp) := \max\{|\wv| : \bp^{\wv} \neq 0\}$ and
$\deg_{\word{i}}(\bp) := \max\{|\wv|_{\word{i}} : \bp^{\wv} \neq 0\}$,
with the convention $\deg(0) = \deg_{\word{i}}(0) = -\infty$.

The extended tensor algebra is an algebra under the product defined component-wise by
$$
(\bp \otimes \bell)_n
  := \sum_{k=0}^{n} \bp_k \otimes \bell_{n-k}\,,
  \quad n \geq 0\,.
$$
The iterated-integral structure of the signature is naturally compatible
with this product through Chen's identity:
\begin{equation}\label{eq:chen}
\sigXhat[s,u] \otimes \sigXhat[u,t] = \sigXhat[s,t]\,,
  \quad s \leq u \leq t\,.
\end{equation}

\paragraph{Linear functionals of the signature and shuffle product.}
Another key property of path signatures is their ability to linearize
products of linear functionals, as a consequence of integration by parts.
We introduce the pairing
$$
\bracket{\bp}{\sigX[]}
  := \sum_{\wv \in V} \bp^{\wv}\, \sigX[]^{\wv}\,,
  \quad \bp\,,\, \sigX[] \in \eTA[2]\,,
$$
which is well-defined whenever $\bp$ or $\sigX[]$ has finite degree, or
more generally when the seminorm
\begin{equation}\label{eq:def_seminorm}
\|\bp\|_{\sigX[]}
  := \sum_{n \geq 0}
     \biggl|\sum_{|\wv| = n} \bp^{\wv}\, \sigX[]^{\wv}\biggr|
\end{equation}
is finite. The time-augmented signature $\sigXhat[s,t]$ then almost surely
belongs to the set of \textit{group-like elements}:
$$
G := \bigl\{\sigX[] \in \eTA[2]
  \colon\ \sigX[]^{\emptyword} = 1
  \;\;\text{and}\;\;
  \bracket{\wv}{\sigX[]}\, \bracket{\word{w}}{\sigX[]}
    = \bracket{\wv \shuprod \word{w}}{\sigX[]}\,,\;\;
  \wv, \word{w} \in V\bigr\}\,,
$$
where $\wv \shuprod \word{w} \in \eTA[2]$ denotes the \textit{shuffle
product} of $\wv$ and $\word{w}$, a generalization of the Cauchy product
defined as follows.

\begin{definition}
The shuffle product is defined recursively by
$$
\word{vi} \shuprod \word{wj}
  := (\word{v} \shuprod \word{wj}) \otimes \word{i}
   + (\word{vi} \shuprod \word{w}) \otimes \word{j}\,,
$$
for words $\word{v}, \word{w} \in V$ and letters
$\word{i}, \word{j} \in \{\word{0}, \word{1}\}$, with the initialization
$\wv \shuprod \emptyword = \emptyword \shuprod \wv = \wv$.
\end{definition}

Concretely, $\wv \shuprod \word{w}$ is the sum of all interleavings of
$\wv$ and $\word{w}$ that preserve the internal ordering of each word. For
instance, $\word{10} \shuprod \word{1} = \word{101} + \word{110} + \word{110}\,$.

Extending the shuffle product bilinearly to $\eTA[2]$ via
$$
\bp \shuprod \bell
  := \sum_{\wv, \word{w} \in V}
     \bp^{\wv}\, \bell^{\word{w}}\, (\wv \shuprod \word{w})\,,
$$
the group-like property yields
$\bracket{\bp}{\sigX[]}\, \bracket{\bell}{\sigX[]}
  = \bracket{\bp \shuprod \bell}{\sigX[]}$
for $\sigX[] \in G$, whenever $\|\bp\|_{\sigX[]}$ and
$\|\bell\|_{\sigX[]}$ are finite. This follows from the seminorm
\eqref{eq:def_seminorm} being shuffle-compatible:
\begin{equation}
\label{eq:shuffle_compatibility}
\|\bp \shuprod \bell\|_{\sigX[]}
  \leq \|\bp\|_{\sigX[]}\, \|\bell\|_{\sigX[]},
\end{equation}
see \cite*[Section~4.1]{cuchiero2025signaturesdesaffinepolynomial}.
By iterating the shuffle product, we define the \textit{shuffle powers}
$\bp{\shupow{0}} := \emptyword$ and
$\bp{\shupow{n}} := \bp{\shupow{n-1}} \shuprod \bp$, as well as the
\textit{shuffle exponential}. For
$\bp \in \eTA[2]$ with $\bp^{\emptyword} = 0$, it is given by $
\shuexp{\bp}
  := \sum_{n \geq 0}  \bp{\shupow{n}} / n!\,$.
The condition $\bp^{\emptyword} = 0$ ensures that for any fixed word
$\wv$, only finitely many terms contribute to $\shuexp{\bp}^{\wv}\,$. The definition extends to general
elements: for any $\bp \in \eTA[2]$,
$$
\shuexp{\bp} := \exp(\bp^{\emptyword})\, \shuexp{\overline{\bp}}\,,
\quad \text{where} \quad
\overline{\bp} := \bp - \bp^{\emptyword}\, \emptyword\,.
$$

\paragraph{Convergence classes.} Since linear functionals of the signature
generally involve infinite sums, their convergence requires care. Throughout the paper, we focus on evaluating such functionals at group-like elements. We introduce
$$
M_{\bp}(\sigX[])
  := \max_{1 \leq |\wv| \leq \deg(\bp)} |\sigX[]^{\wv}| \,, \quad \bp \in \eTA[2]\,,
$$
as a measure of how far $\sigX[]$ truncated at order $\deg(\bp)$ is from $\emptyword\,$, and the corresponding ball of radius $r > 0$
in $G$:
\begin{equation}
\label{eq:def_G_p_r}
G_{\bp, r}
  := \bigl\{\sigX[] \in G
     \,:\, M_{\bp}(\sigX[]) \leq r\bigr\}\,.
\end{equation}
An element satisfying $\|\bp\|_{\sigX[]} < +\infty$ for all
$\sigX[] \in G_{\bp, r}$ thus yields a linear functional
with finite radius of convergence $r$.

\paragraph{Left and right shifts.} For a power series
$\sum_{n \geq 0} \bp_n x^n / n!\,$, differentiation with respect to $x$
shifts the coefficient sequence $(\bp_0, \bp_1, \ldots)$ into
$(\bp_1, \bp_2, \ldots)$.  A natural extension of such coefficient shifts to the extended tensor algebra are
the \textit{right} and \textit{left shifts} defined by
\begin{equation}
\label{eq:def_shifts}
\bp|_{\word{v}}
  := \sum_{\word{w} \in V} \bp^{\word{wv}}\, \word{w}\,, \quad \text{and} \quad {}_{\word{v}}|\bp
  := \sum_{\word{w} \in V} \bp^{\word{vw}}\, \word{w}\,,
  \quad \word{v} \in V\,.
\end{equation}
These operators appear naturally as differentiation operators when
applying It\^o's formula to $\bracket{\bp}{\sigXhat[t]}$ (see \cite*[Theorem 3.5]{riccatifourierlaplace}). By linearity, both extend to shifts by elements of the extended tensor
algebra:
\begin{equation}
\label{eq:def_X_shifts}
\bp|_{\sigX[]}
  := \sum_{\word{v} \in V} \sigX[]^{\word{v}} \cdot \bp|_{\word{v}}
   = \sum_{\word{v} \in V}
     \bracket{{}_{\word{v}}|\bp}{\sigX[]} \cdot \word{v}\,,
\quad \text{and} \quad
{}_{\sigX[]}|\bp
  := \sum_{\word{v} \in V} \sigX[]^{\word{v}} \cdot {}_{\word{v}}|\bp
   = \sum_{\word{v} \in V}
     \bracket{\bp|_{\word{v}}}{\sigX[]} \cdot \word{v}\,,
\end{equation}
which are well-defined whenever
$\|\bp|_{\word{v}}\|_{\sigX[]} < +\infty$ or
$\|{}_{\word{v}}|\bp\|_{\sigX[]} < +\infty$ for all $\wv \in V$,
respectively.

If $\bp$ has finite degree and
$\sigX[], \sigY[] \in \eTA[2]$, then
\begin{equation}
\label{eq:recentering_bracket}
\bracket{\bp}{\sigX[] \otimes \sigY[]}
  = \bracket{{}_{\sigX[]}|\bp}{\sigY[]}
  = \bracket{\bp|_{\sigY[]}}{\sigX[]}\,.
\end{equation}
This shows that left and right shifts are dual to left and right tensor
multiplication. Moreover, \eqref{eq:recentering_bracket} is the tensor algebra analogue of
the recentering of polynomials  or power series; see \cite*[Remark 3.2]{riccatifourierlaplace}.

\section{Riccati equations on the extended tensor algebra}\label{section:riccati_prelim}

In this section, we introduce our admissible class of reward functions for
the control problem  and discuss the existence of the solution to the Riccati equation on the extended tensor algebra. 

We seek reward functionals that are linear functionals
of the time-augmented signature,
$F(X) = \bracket{\bp}{\sigXhat[T]}$, and define a class
$\mathcal{B}$ of admissible coefficients $\bp \in \eTA[2]$ for which the value of the optimal control problem \eqref{eq:introobjective} is finite. Thanks to the Boué--Dupuis formula \eqref{eq:introBD}, this question reduces to the existence of the Laplace transform $\E[\exp(\bracket{\bp}{\sigW[T]})]$. For this problem, an appropriate class $\mathcal{B}$ was introduced in \cite*[Section 4.1]{riccatifourierlaplace} as the
one allowing the conditional Laplace transform
$\E[\exp(\bracket{\bp}{\sigW[T]})\, |\, \mathcal{F}_t]$ and its logarithm
to be finite and expandable as linear functionals of the running signature
$\sigW[t]$. Consequently, it is also the natural class of admissible
terminal conditions ensuring existence of a solution to the signature
Riccati equation.

The main ideas underpinning the class $\mathcal{B}$ are the following.
First, we restrict to $\bp$ with finite degree, though the degree can be
arbitrarily large. Second, the class is built around two dominant terms
that control the remaining lower-degree contributions, in the spirit of
polynomials whose asymptotic behaviour is governed by their leading term.
Finally, the identification of this class was made possible by specific
estimates for time-augmented signatures of one-dimensional {continuous semimartingales} 
derived in \cite*{riccatifourierlaplace}, based on a powerful algebraic structure, namely that the shuffle algebra is generated by \cite*{Lyndon1954}  words due to the \cite{Radford1979ANR} theorem.

We now state our definition of the class $\mathcal{B}$, which matches the
one in \cite*{riccatifourierlaplace} restricted to real-valued coefficients.
\begin{definition}\label{def:class_B}
We say that $\bp \in \mathcal{B} \subset T((\R^2))$ if there exist three
real numbers $a, b \geq 0$, $c \in \R$, two natural numbers
$m, n \in \N$, and three coefficients
$\bq, \br, \bm{s} \in T((\R^2))$ of finite degree, satisfying
\begin{align}\label{eq:class_B_coefs_cond}
  \deg_{\word{1}}(\bq) < 2n\,, \quad
  \deg_{\word{0}}(\bq) = 0\,, \quad
  \deg_{\word{1}}(\br) < 2m\,, \quad
  \deg_{\word{1}}(\bm{s}) < 2m \land 2n\,,
\end{align}
such that
\begin{equation}\label{eq:class_B_p}
  \bp = -a(\word{1}\conpow{2n} + \bq)
      - b(\word{1}\conpow{2m} + \br)\word{0}
      + ab\cdot\bm{s}
      + c \emptyword\,.
\end{equation}
The time degrees $\deg_{\word{0}}(\br)$ and $\deg_{\word{0}}(\bm{s})$ are
only required to be finite.
\end{definition}

We now analyse each term on the right-hand side of \eqref{eq:class_B_p} to
provide intuition on the reward functionals it generates.
\begin{itemize}
  \item \textbf{Markovian terminal reward.} The term
    $\bp_1 := -a(\word{1}^{\otimes 2n} + \bq)$ contains only the letter
    $\word{1}$ (since $\deg_{\word{0}}(\bq) = 0$), so
    $\bracket{\bp_1}{\sigXhat[T]} = P_1(X_T)$ is a polynomial in the
    terminal value $X_T$. The assumptions force $P_1$ to have even degree
    and negative leading coefficient, hence to be bounded from above. This is a natural requirement on $P_1$ since otherwise, when $\deg P_1 \geq 3$, the optimal control problem would have infinite value.

  \item \textbf{Running reward.} The term
    $\bp_2 := -b(\word{1}^{\otimes 2m} + \br)\word{0}$ yields an integrated
    functional
    $$\bracket{\bp_2}{\sigXhat[T]}
      = -b\int_0^T \left(\frac{(X_t - X_0)^{2m}}{(2m)!}
        + \bracket{\br}{\sigXhat[t]}\right)\, dt.
    $$
    Unlike the terminal reward, $\br$ may contain both letters $\word{0}$
    and $\word{1}$, making this term very general. The only requirement is
    the leading integrated monomial $X_t^{2m}$ with negative coefficient. As an
    example, if $\deg(\bell) = m$ and
    $\bell^{\word{1}^{\otimes m}} \neq 0$, then
    $\bp = -\bell^{\shupow{2}}\word{0}$, corresponding to
    $\bracket{\bp}{\sigXhat[T]} = -\int_0^T
      \bracket{\bell}{\sigXhat[t]}^2\, dt$,
    satisfies these conditions and belongs to $\mathcal{B}$.

  \item \textbf{General regularized coefficient.} The term $\bm{s}$ is a
    general element of $\eTA[2]$ with finite degree: no specific internal
    structure is imposed on it, unlike $\bq$ and $\br$. Such terms naturally arise when approximating path-dependent functionals by linear functionals of the signature, via the universality property. In general, however, such $\bm{s}$ do not belong to $\mathcal B$. To enforce this admissibility, one introduces the regularized combination
$
-a \cdot\word{1}^{\otimes 2n} - b \cdot\word{1}^{\otimes 2m}\word 0 + ab\cdot \bm{s},$ with $\deg_{\word{1}}(\bm{s}) < 2m \land 2n$ 
whose structure ensures that the associated $\bp$ lies in $\mathcal B$.

\end{itemize}

We note in passing that $\mathcal{B}$ is a convex cone, so that any linear combination with positive coefficients of admissible rewards defines an admissible reward as well.

We now consider the infinite-dimensional Riccati equation, arising from the Laplace transform \eqref{eq:introBD} when taking $F(X) = \bracket{\bp}{\sigX[]}$:
\begin{equation}\label{eq:Riccati_eq}
    \begin{cases}
        \dot{\bpsi}_t + \bpsi_t|_{\word 0} + \frac{1}{2}\bpsi_t|_{\word{11}} + \frac{1}{2}(\bpsi_t|_{\word{1}})\shupow{2} = 0\,, \quad t \in [0\,, T]\,, \\
        \bpsi_T = \bp
    \end{cases}
\end{equation}
We recall the existence result for \eqref{eq:Riccati_eq}
proved in \cite*[Theorems~4.7 and Theorem~5.7]{riccatifourierlaplace}.
\begin{theorem}
\label{thm:results_papier1}
    Let $\bp \in \mathcal B$ and $T > 0\,$. There exists a solution $\bpsi : [0\,, T] \to \eTA[2]\,$ to the Riccati equation \eqref{eq:Riccati_eq}, such that $\bpsi^{\word v} \in C^1([0\,,T], \R)$ for all $\wv \in V\,$, and
    \begin{enumerate}
        \item For all $t \in [0\,,T]\,$, $\varepsilon \in (0\,,1)\,$, and $\sigX[] \in G$ satisfying
        \begin{equation}
        \label{eq:cv_condition}
        \|\shuexp{\overline{\bpsi_t}}\|_{\sigX[]} \leq 2 - \varepsilon \,, \quad \text{where} \quad \overline{\bpsi_t} := \bpsi_t - \bpsi_t^{\emptyword} \emptyword\,,
        \end{equation}
        we have $\|\bpsi_t\|_{\sigX[]} < \infty$ and 
        $$
        \log \E\left[\exp\left(\bracket{\bp}{\sigX[] \otimes \sigW[t,T]}\right) \right] = \bracket{\bpsi_t}{\sigX[]}\,.
        $$
        \item For any $\varepsilon\in(0, 1)$, there exists $r = r(\bp, T, \varepsilon) > 0$ such that \eqref{eq:cv_condition} holds uniformly over $t \in [0\,,T]$ for any $\sigX[] \in G_{\bp, r}\,$; see \eqref{eq:def_G_p_r}.
        \item For any word $\word v \in V$ and $\varepsilon \in (0\,,1)\,$, there exists a function $f_{T, \varepsilon, \word v} : \eTA[2] \times \R_+ \times \R_+ \to \R_+\,$, measurable in its first argument and continuous in the second and the third, such that for all $t \in [0\,,T]$ and $\sigX[] \in G$ satisfying \eqref{eq:cv_condition} with $\varepsilon$, we have 
        \begin{equation}
        \label{eq:bound_psi_proj_v}
        \|\bpsi_t|_{\word v}\|_{\sigX[]} \leq f_{T, \varepsilon, \word v}(\bp, t, M_{\bp}(\sigX[]))\,.
        \end{equation}
    \end{enumerate}
\end{theorem}
\begin{proof}
    The first point is direct from \cite*[Theorem 4.7, Theorem 5.7]{riccatifourierlaplace}. 

    The proof of \cite*[Theorem 4.7]{riccatifourierlaplace} gives 
    $$
    f_{T, \varepsilon, \emptyword}(\bp, t, x) = |\bpsi_t^{\emptyword}| + \log(\varepsilon^{-1})\,,
    $$
    which is continuous in $t$ from the third point of the same theorem. For $\wv \neq \emptyword\,$, the proof of \cite*[Theorem 5.7]{riccatifourierlaplace} yields 
    $$
    f_{T, \varepsilon, \wv}(\bp, t, x) = C'_{\bp, T, \wv} \left((1 + \varepsilon^{-1} \exp(- \bpsi_t^{\emptyword}))\,(1 + x) \,\exp(C_{\bp, T}(1 + x)^{\deg_{\word 1}(\bp)}) \right)^{|\wv|}\,,
    $$
    where the constants $C_{\bp, T, \wv}'$ and $C_{\bp, T}$ can be chosen to be measurable functions of $\bp\,$, as they are derived from basic inequalities throughout the paper. The function $f_{T, \varepsilon, \wv}$ thus also satisfies the assumptions of Theorem \ref{thm:results_papier1}, which finishes the proof.
\end{proof}

Theorem~\ref{thm:results_papier1} establishes only a local expansion of the log-Laplace transform, hence proving its analyticity, but not that it is entire. This result is sharp in the sense that in general one cannot expect the log-Laplace transform, closely related to the HJB equation, to be entire; see \cite*[Section 4.4]{riccatifourierlaplace} for a more detailed discussion. 
For the optimal control problem, one expects to be able to expand the value function only locally around the current signature of the controlled process.
Indeed, at time $t$, given the signature $\sigXahat$ of the realized control variable path, one can formally apply the Boué--Dupuis formula to obtain the value-to-go
$$
\sup_{\alpha \in \mathcal{A}_{t, T}}
\mathbb{E}
\Big[
- \frac 1 2  \int_t^T \alpha_s^2\,ds
+
\bracket{\bp}{\sigXahat\otimes\sigXahat[t,T]}
\Big]
=
\log
\mathbb{E}
\Big[
\exp({\bracket{\bp}{\sigXahat\otimes\sigW[t,T]}})
\Big]. 
$$
Here $\mathcal{A}_{t, T}$ denotes the set of admissible controls between $t$ and $T$, and $\sigW[t,T]$ the time-augmented signature of a standard Brownian motion, over which the expectation is taken. Theorem~\ref{thm:results_papier1} then states that the signature expansion of the log-Laplace transform on the right-hand side converges if $\|\shuexp{\overline{\bpsi_t}}\|_{\sigXahat[t]} \leq 2 - \varepsilon$. This, in turn, is true only when $\sigXahat[t]$ is close to the trivial signature $\emptyword$, which is, of course, not always the case. To overcome this difficulty, one can choose $t_0 \leq t$ such that $\sigXastarhat[t_0, t]$ is small enough, and use Chen's identity along with the recentering property \eqref{eq:recentering_bracket} to obtain
$$
\log
\mathbb{E}
\Big[
\exp({\bracket{\bp}{\sigXahat\otimes\sigW[t,T]}})
\Big] =
\log
\mathbb{E}
\Big[
\exp({\bracket{\bp}{\sigXahat[t_0]\otimes\sigXahat[t_0, t]\otimes\sigW[t,T]}})
\Big] =
\log
\mathbb{E}
\Big[
\exp({\bracket{{}_{\sigXahat[t_0]}|\bp}{\sigXahat[t_0, t]\otimes\sigW[t,T]}})
\Big],
$$
which reduces to the initial log-Laplace transform problem with $\sigXahat[t]$ replaced by $\sigXahat[t_0, t]$ and $\bp$ replaced by ${}_{\sigXahat[t_0]}|\bp$. To apply Theorem~\ref{thm:results_papier1} again, one should only ensure that the shifted coefficient ${}_{\sigXahat[t_0]}|\bp$ is admissible, which is true by virtue of the following proposition \cite*[Lemma 4.4]{riccatifourierlaplace}.
\begin{proposition}\label{prop:left-proj-compatibility}
    If $\bp \in \mathcal{B}$ and $\sigX[] \in T((\R^2))$ such that $\sigX[]^{\emptyword} = 1$, then $~_{\sigX[]}|\bp \in \mathcal{B}$.
\end{proposition}
Theorem~\ref{thm:results_papier1} yields the following expansion 
$$
\sup_{\alpha \in \mathcal{A}_{t, T}}
\mathbb{E}
\Big[
- \frac 1 2  \int_t^T \alpha_s^2\,ds
+
\bracket{\bp}{\sigXahat[T]}
\Big]
=
\log
\mathbb{E}
\Big[
\exp({\bracket{\bp}{\sigXahat\otimes\sigW[t,T]}})
\Big] = \bracket{\bpsi_t^{\sigXahat[t_0]}}{\sigXahat[t_0,t]},
$$
where, for $\mathbb{X} \in G$, $\bpsi_t^{\sigX[]}$ denotes the solution to the recentered Riccati equation 
\begin{equation}\label{eq:Riccati_recentered}
    \begin{cases}
\dot{\bpsi}_t^{\sigX[]} + \bpsi^{\sigX[]}_t|_{\word 0} + \frac{1}{2} \bpsi^{\sigX[]}_t|_{\word{11}} + \frac{1}{2}(\bpsi^{\sigX[]}_t|_{\word 1})\shupow{2} = 0 \\
        \bpsi^{\sigX[]}_T = {}_{\sigX[]}|\bp.
\end{cases}    
\end{equation}
These ideas are made precise in the next section.

\section{Optimal control and value process}\label{section:oc}

In this section, we establish the representations of the optimal control and the value process as linear expansions in the signature of the controlled variable. Since our construction of the optimal control as the Föllmer drift relies on a change-of-measure argument, it is more convenient to work with the weak formulation of the stochastic optimal control problem.
Consider the canonical filtered probability space $(\Omega, \F, \mathbb{F}, \P)$, where
$$
\Omega = \{\omega \in C([0, T], \R)\colon\ \omega(0) = 0\},
$$
and let $X = (X_t)_{t \in [0, T]}$ be the coordinate process: $X_t(\omega) = \omega(t)$. Here $\P$ denotes Wiener measure, $\mathbb{F} = (\F_t)_{t \in [0, T]}$ is the augmented canonical filtration, and $\F = \F_T$.
We define the class of admissible weak controls as follows.

\begin{definition}
    An admissible weak control is a tuple
    $$
    \pi = (\Q, W^{\Q}, \alpha),
    $$
    such that $(\Omega, \F, {\mathbb{F}}, \Q)$ is a filtered probability space satisfying the usual conditions, $W^\Q$ is a $\Q$-Brownian motion, and $\alpha$ is ${\mathbb{F}}$-adapted with $\E^{\Q}[\int_0^T\alpha_t^2\,dt] < \infty$. For an admissible weak control $\pi$, the controlled variable $X^\alpha$  is defined by 
    \begin{equation}\label{eq:controlled_var_def_weak}
        X_t^\alpha := W^\Q_t + \int_0^t\alpha_s\, ds.
    \end{equation}
    The class of admissible controls is denoted by $\mathcal{A}$.
\end{definition}

In what follows, we use $\pi$ and $\alpha$ interchangeably.

{
Let $\bp \in T((\R^2))$ and define the following stochastic optimal control problem

\begin{equation}\label{eq:weak_primal_problem}
    \sup_{\alpha \in \mathcal{A}}J(\alpha) = \sup_{\alpha \in \mathcal{A}}\E^{\Q}\left[ - \frac12\int_0^T\alpha_t^2\,dt + \bracket{\bp}{\sigXahat[T]} \right].
\end{equation}

\begin{definition}
    The value process corresponding to \eqref{eq:weak_primal_problem} is defined by
    $$
    V_t^* = \esssup_{\alpha \in \mathcal{A}_{t, T}} \E^{\Q}\left[  - \frac12\int_t^T\alpha_s^2\,ds + \bracket{\bp}{\sigXahat[T]} \, \Bigg|\, \F_t\right],
    $$
    where $\mathcal{A}_{t, T}$ denotes the set of ${\mathbb{F}}$-adapted processes $\alpha$ on $[t, T]$ with $\E^{\Q}[\int_t^T\alpha_s^2\,ds] < \infty$.
\end{definition}
}

\subsection{Main result}

We now state our main theorem that establishes the existence of an optimal control with a feedback representation in the infinite dimensional time-augmented signature $\widehat {\mathbb X}^{\alpha^*}$ of the controlled variable. More precisely,  the value process and the optimal control admit a local signature expansion whose coefficients solve the recentered Riccati equation \eqref{eq:Riccati_recentered} on the extended tensor algebra.

\begin{theorem}
\label{thm:oc_sig_repr}
    Let $\bp \in \mathcal B\,$. Then, there exists a weak optimal control $\alpha^* \in \mathcal A$, i.e.~$J(\alpha^*)=\sup_{\alpha \in \mathcal A} J(\alpha)$. Furthermore, let $\varepsilon \in (0\,,1)\,$ and   $\sigma$ be a stopping time such that $\sigma \in [0,T]$ a.s. Then,  there exists a solution  ${{\bpsi}^{\sigXastarhat[\sigma]}}$ to the Riccati equation \eqref{eq:Riccati_recentered} with $\mathbb X =\sigXastarhat[\sigma] $ such that 
 the optimal value process $V^*$ and the optimal control $\alpha^*$ are given by
\begin{equation}\label{eq:value_control_local}
        V^*_t = \bracket{\bpsi_t^{\sigXastarhat[\sigma]}}{\sigXastarhat[\sigma, t]},\ t \in \llbracket \sigma, \tau_{\sigma} \rrbracket, \quad \text{and a.s.} \quad \alpha^*_t = \bracket{\bpsi_t^{\sigXastarhat[\sigma]}|_{\word 1}}{\sigXastarhat[\sigma, t]},\ \text{for a.e. $t \in \llbracket \sigma, \tau_{\sigma} \rrbracket\,$,}
    \end{equation}
    where
    \begin{equation}\label{eq:stopping_time_def}
        \tau_{\sigma} := \inf \left\{t \geq \sigma\,:\, \| \exp^{\shuprod}({\overline{\bpsi}_t^{\sigXastarhat[\sigma]}}) \|_{\sigXastarhat[\sigma,t]} \geq 2 - \varepsilon \right\} \wedge T\,.
    \end{equation}
We also  have $\tau_{\sigma} > \sigma$ almost surely.
\end{theorem}

\begin{sqremark}    Theorem~\ref{thm:oc_sig_repr} provides an explicit form of the functional $\Phi_t$ appearing in Lemma~\ref{lemma:generic_pathdep_control}, hence showing that the optimal control is closed-loop with an explicit signature-Markovian feedback functional, computable via the Riccati equation.
\end{sqremark}

{\begin{sqremark}
    As in Section~\ref{section:riccati_prelim}, the representation of the optimal control, as well as the value process is only local with finite convergence radius. In particular, when $\sigma = 0$, \eqref{eq:value_control_local} corresponds to the expansion around the trivial signature $\emptyword$, and the recentered Riccati equation reduces to the standard one given by \eqref{eq:Riccati_eq}. Global representation is then recovered via the recentering technique similar to that explained in the end of Section~\ref{section:riccati_prelim}. The dynamic recentering algorithm ensuring the convergence at each time $t$ is presented in Subsection~\ref{subsection:tracking}.
\end{sqremark}
}

{\begin{corollary}[Running reward]\label{C:running}
   Assume that the coefficient $\bp\in\mathcal{B}$ is decomposed as $\bp =\bm f\word{0} + \bm g $. The objective function can then be written as a sum of the running reward with coefficient $\bm f$ and the terminal reward with coefficient $\bm g$: 
    $$
    \sup_{\alpha\in\mathcal{A}}\,
\E^\Q\left[
-\frac{1}{2}\int_0^T \alpha_t^2\,dt
 + \int_0^T \langle \bm f, \widehat{\mathbb X}^{\alpha}_s  \rangle \, ds + \langle \bm g, \widehat{\mathbb X}^{\alpha}_T  \rangle 
\right].
$$ 
With the running reward made explicit, the value process is defined by
$$
\widetilde{V}_t^* = \esssup_{\alpha \in \mathcal{A}_{t, T}} \E^{\Q}\left[  - \frac12\int_t^T\alpha_s^2\,ds + \int_t^T\bracket{\bm f}{\sigXahat[s]}\,ds + \bracket{{\bm g}}{\sigXahat[T]} \, \Bigg|\, \F_t\right],
$$
and the optimal control has the representation \eqref{eq:value_control_local} with the coefficient $\bchi^{\sigX[]}$, {instead of  $\bpsi^{\sigX[]}$,} solving the Riccati equation with a source term:
\begin{equation}\label{eq:Riccati_recentered_source}
\begin{cases}
    \dot{\bm \chi}^{\sigX[]}_t + \bchi^{\sigX[]}_t|_{\word 0} + \frac{1}{2} \bchi^{\sigX[]}_t|_{\word{11}} + \frac{1}{2}(\bchi^{\sigX[]}_t|_{\word 1})\shupow{2} + ~_{\sigX[]}|{\bm f} = 0, \quad t \in [0, T]\\
    \bm \bchi_T^{\sigX[]} = ~_{\sigX[]}|{\bm g}.
\end{cases}
\end{equation}
The stopping time $\tau_\sigma$ is defined by \eqref{eq:stopping_time_def} with $\bpsi^{\sigX[]}$ replaced by $\bchi^{\sigX[]}$.
\end{corollary}
\begin{proof}
    Note that we have the following relationship between the value processes $V^*$ and $\widetilde{V}^*$ from Theorem~\ref{thm:oc_sig_repr}:
    \begin{equation}\label{eq:value_processes_link}
        V_t^* = \widetilde{V}^*_t + \int_0^t\bracket{{\bm f}}{\sigXastarhat[s]}\,ds = \widetilde{V}^*_t + \bracket{{\bm f}\word{0}}{\sigXastarhat[t]}.
    \end{equation}
    Set $\bchi_t^{\sigX[]} := \bpsi_t^{\sigX[]} - ~_{\sigX[]}|({\bm f}\word{0})$, where $\bpsi^{\sigX[]}$ satisfies \eqref{eq:Riccati_recentered}. Since $\bchi_T = \bpsi_T - ~_{\sigX[]}|({\bm f}\word{0}) = ~_{\sigX[]}|{\bm g}$, $\dot\bpsi_t = \dot\bchi_t$, and
    $$
    \bpsi^{\sigX[]}_t|_{\word 0} + \frac{1}{2} \bpsi^{\sigX[]}_t|_{\word{11}} + \frac{1}{2}(\bpsi^{\sigX[]}_t|_{\word 1})\shupow{2} = \bchi^{\sigX[]}_t|_{\word 0} + \frac{1}{2} \bchi^{\sigX[]}_t|_{\word{11}} + \frac{1}{2}(\bchi^{\sigX[]}_t|_{\word 1})\shupow{2} + ~_{\sigX[]}|{\bm f},
    $$
    \eqref{eq:Riccati_recentered_source} holds. Finally, we observe that 
    $$
    \widetilde{V}_t^* = \bracket{\bpsi^{\sigXastarhat[\sigma]}}{\sigXastarhat[\sigma, t]}- \bracket{{\bm f}\word{0}}{\sigXastarhat[t]} = \bracket{\bpsi^{\sigXastarhat[\sigma]} -~_{\sigXastarhat[\sigma]}|({\bm f}\word{0})}{\sigXastarhat[\sigma, t]} = \bracket{\bchi_t^{\sigXastarhat[\sigma]}}{\sigXastarhat[\sigma, t]}.
    $$
    The representations of the optimal control follow from the fact that $\bchi_t^{\sigX[]}\proj{1} = \bpsi_t^{\sigX[]}\proj{1}$.

    The convergence condition $ \| \exp^{\shuprod}({\overline{\bchi}_t^{\sigXastarhat[\sigma]}}) \|_{\sigXastarhat[\sigma,t]} \leq 2 - \varepsilon$ arises from the same principles as the one for $\bpsi_t^{\sigXastarhat[\sigma]}$: indeed, as in \cite[Remark 5.1]{riccatifourierlaplace}, one can show that $\tilde\bu_t= \exp^{\shuprod}({{\bchi}_t^{\sigXastarhat[\sigma]}})$ is a solution to a linear differential equation on the tensor algebra corresponding to an entire linear functional of the signature. The coefficient ${\bchi}_t^{\sigXastarhat[\sigma]}$ is hence defined as a shuffle logarithm of $\tilde\bu_t$, so that the condition $ \| \exp^{\shuprod}({\overline{\bchi}_t^{\sigXastarhat[\sigma]}}) \|_{\sigXhat[]} \leq 2 - \varepsilon$ ensures the convergence of $\bracket{\bchi_t^{\sigXastarhat[\sigma]}}{\sigXhat[]}$ via shuffle-compatibility, see the proof of \cite[Theorem 4.7]{riccatifourierlaplace}.
\end{proof}

\begin{remark}\label{rmk:timedep_running}
We expect our methodology to extend naturally to time-dependent coefficients, that is, by replacing the running reward $\langle {\bm f}, \widehat{\mathbb X}^{\alpha}_s \rangle$ in Corollary~\ref{C:running} with $\langle {\bm f}_s, \widehat{\mathbb X}^{\alpha}_s \rangle$, where $({\bm f}_s)_{s \in [0, T]}$ belongs to some reasonable class, for instance, continuous in time and of the form
$$
\bm f_s = -a\word{1}^{\otimes 2n} + \tilde{\bm f}_s, \quad a \geq 0,\quad n \in \N,\quad \deg_{\word{1}}(\tilde{\bm f}_s) < 2n.
$$
In this case, we conjecture that the optimal control retains the same feedback form as in \eqref{eq:value_control_local}, with $\bchi^{\sigX[]}$ solving the Riccati equation \eqref{eq:Riccati_recentered_source} with time-dependent source $(~_{\sigX[]}|{\bm f}_t)_{t\in[0,T]}$. Establishing this extension would require a corresponding generalization of the log-Laplace transform results of \cite*{riccatifourierlaplace} to the time-dependent setting. Such an extension could be particularly useful for representing certain non-Markovian models within our framework while reducing dimensional complexity, as it would avoid the need to expand in the time variable, for instance in the case of Volterra structures, see for instance Section~\ref{sect:volterra}  and \cite*[Section 4.3]{abi2024path}. Our formulas in the time-dependent case are validated numerically in Figure~\ref{fig:control_timedep_mu} below, which  supports our conjecture.
\end{remark}

\begin{sqremark}[{The Föllmer drift via Malliavin calculus for signatures}]
    In the signature control problem with the  terminal reward $\bracket{\bp}{\sigXahat[T]}$, $\bp \in \mathcal{B}$, the feedback control functional can be equivalently written in terms of expectations of functionals of the Brownian signature. Indeed, recalling the Clark--Ocone representation \eqref{eq:clark_ocone_repr} of the optimal control which follows from the proof of Lemma~\ref{lemma:generic_pathdep_control} below, and using the fact that the Malliavin derivative of the Brownian signature is given by
    $$
    \mathbf{D}_t\sigW[T] = \sigW\otimes\word{1}\otimes\sigW[t,T],
    $$
    as shown in \cite*[Corollary 4.7]{jaber2026malliavincalculussignaturesapplications}, we obtain that the value function and the optimal control are given by the following signature feedback maps:
\begin{equation}\label{eq:value_fct_control_mc}
        v(t, \sigXhat) = \log\E^{\Q^*}[\exp(\bracket{\bp}{\sigXhat\otimes\sigW[t, T]^{\Q^*}})], \qquad \alpha^*(t, \sigXhat) = \dfrac{\E^{\Q^*}[\exp(\bracket{\bp}{\sigXhat\otimes\sigW[t, T]^{\Q^*}})\bracket{\bp}{\sigXhat\otimes\word{1}\otimes\sigW[t,T]^{\Q^*}}]}{\E^{\Q^*}[\exp(\bracket{\bp}{\sigXhat\otimes\sigW[t, T]^{\Q^*}})]},
    \end{equation}
    reducing the simulation to that of the Brownian signature $\sigW[t, T]^{\Q^*}$ only. The disadvantage of this method is the necessity to run the Monte Carlo simulation \eqref{eq:value_fct_control_mc} at each time $t \in [0, T]$, while the Riccati approach of Theorem~\ref{thm:oc_sig_repr} allows one to solve the Riccati equation once at some $\sigma \in [0, T]$ to determine the coefficients $(\bpsi^{\sigXhat[\sigma]}_t)_{t \in [\sigma, T]}$, which can be used as long as the convergence condition is satisfied. The Monte Carlo approximations of the  representations \eqref{eq:value_fct_control_mc} are used as benchmarks in our numerical experiments in Figures~\ref{fig:control_constant_mu}-\ref{fig:control_timedep_mu} below.
\end{sqremark}

{\begin{corollary}[Markovian case]\label{C:markovian}
    Let $f, g:\mathbb R \to \mathbb R$  be polynomials
    $$ f(x) = \sum_{k=0}^M f_k \frac{x^k}{k!}, \qquad g(x) = \sum_{k=0}^N g_k \frac{x^k}{k!},$$
for some $f_k, g_k \in \mathbb R$, with even degrees $N, M$ and leading coefficients $f_M, g_N < 0$.  Then, for any $x_0 \in  \R$, there exists $r_{x_0}>0$ such that   the value  function
{\begin{align}\label{eq:markovianvalue}
\tilde v(t,x):=\sup_{\alpha\in\mathcal A_{t,T}}\mathbb E^{\mathbb Q}\left[
-\frac12\int_t^T\alpha_s^2 ds + \int_t^Tf(X_s^\alpha)ds+g(X_T^\alpha)\, \Bigg| \, X_t^\alpha = x
\right],    
\end{align}
and the optimal control map $\alpha^*(t,X^{\alpha_t^*}_t)$ admit the power series representations:
\begin{align}\label{eq:markovianalpha}
\tilde v(t,x)=\sum_{n\geq 0} \bchi^{x_0,n}_{t} \frac{(x-x_0)^n}{n!}, \quad \alpha^*(t,x)=\sum_{n\geq 0} \bchi^{x_0,n+1}_{t} \frac{(x-x_0)^n}{n!}, \quad |x-x_0| \leq r_{x_0},
\end{align}
}
where the family $t \mapsto \bchi_t^{x_0} = (\bchi^{x_0,n}_{t})_{n\geq 0}$ solves the infinite-dimensional Riccati equation 
{\begin{equation}\label{eq:Riccati_markovian}
\begin{cases}
\dot{\bchi}^{x_0, n}_t +\dfrac12\bchi^{{x_0}, n+2}_{t}
+\dfrac12\displaystyle\sum_{k=0}^n \binom{n}{k}\bchi^{x_0, k+1}_{{t}}\bchi_{t}^{x_0, n-k+1} + f^{(n)}(x_0)=0, \quad t \in [0, T], \quad n \geq 0,\\
\bchi_{T}^{{x_0}, n}=g^{(n)}(x_0),
\qquad n\geq 0,
\end{cases}
\end{equation}
}
where $g^{(n)}$ denotes the $n$-th derivative of $g$.
\end{corollary}

\begin{proof}
 This is an immediate application of Corollary~\ref{C:running} with $\bm f = \sum_{k=0}^M f_k \word{1}^{\otimes k}$ and ${\bm g} = \sum_{k=0}^N g_k \word{1}^{\otimes k}$ and the fact that for power series, the shuffle product reduces to the Cauchy product and the shifts correspond to differentiations.
\end{proof}

\begin{remark}[Link with the HJB equation]\label{rmq:hjb}
In the Markovian case, with running reward $f(X_s^\alpha)$ and terminal reward $g(X_T^\alpha)$, the value function $\tilde v$ given by \eqref{eq:markovianvalue}
 solves the HJB equation
$$
\partial_t \tilde v(t,x)+\frac12\partial_{xx}\tilde v(t,x)+\frac12|\partial_x \tilde v(t,x)|^2+f(x)=0,
\qquad \tilde v(T,x)=g(x),
$$
and the optimal feedback control is given by $
\alpha_t^*=\partial_x \tilde v(t,X_t^{\alpha^*}).
$ Corollary \ref{C:markovian} shows  that in the case of polynomial  functions $f$ and $g$ the solution to the HJB equation above is analytic and admits a power series expansion around any $x_0 \in \R$ with time-dependent coefficients 
satisfying the infinite Riccati system. The optimal control \eqref{eq:markovianalpha}  then follows by differentiating $v(t, x)$ with respect to $x$. {Although efficient alternatives to the analytic series expansion of the HJB solution exist in the Markovian setting, the genuinely path-dependent case is fundamentally different. There, the value function satisfies a path-dependent HJB equation, and the analytic signature expansion provides a tractable means of approximating its solution, whereas a direct numerical solution of the path-dependent PDE is infeasible due to the infinite-dimensional state space.}
\end{remark}

{\begin{sqremark}[Linear-Quadratic case]\label{Ex:LQ}
When $N, M \leq 2$ in Corollary~\ref{C:markovian}, we have that $f(x) = f_0 + f_1x + f_2x^2/2$ and $g(x) = g_0 + g_1x + g_2x^2/2$, so that  the Riccati equation \eqref{eq:Riccati_markovian} is finite-dimensional:
\begin{equation*}
\begin{cases}
    \dot\bchi_t^0 + \dfrac{1}{2}\bchi_t^2 + \dfrac{1}{2}(\bchi_t^1)^2 + f(0) = 0, & \bchi_T^0 = g(0), \\
    \dot\bchi_t^1 + \bchi_t^1\bchi_t^2 + f'(0) = 0, & \bchi_T^1 = g'(0), \\
    \dot\bchi_t^2 + (\bchi_t^2)^2 + f''(0) = 0, &\bchi_T^2 = g''(0).
\end{cases}
\end{equation*}
We recover the standard LQ setting with a quadratic value function and linear optimal control maps:
$$
\widehat{v}(t, x) = \bchi_t^0 + \bchi_t^1 x + \bchi_t^2\frac{x^2}{2}, \qquad \alpha^*(t, x) = \bchi_t^1 + \bchi_t^2 x,
$$
which, of course, hold globally for all $x \in \R$.
\end{sqremark}
}
}

\subsection{Proof of the main result}
Before proceeding with the proof of Theorem~\ref{thm:oc_sig_repr}, we recall the construction of the optimal control as the Föllmer drift and derive the representation of the value process for the generic path-dependent control problem
$$
    \sup_{\alpha \in \mathcal{A}}J(\alpha) = \sup_{\alpha \in \mathcal{A}}\E^{\Q}\left[ - \frac12\int_0^T\alpha_t^2\,dt + F(X^\alpha) \right],
$$
where $F$ is some measurable real-valued function of $(X_t)_{t \in [0, T]}$. The construction of $\alpha^*$ follows the proof of \cite*[Theorem 2]{Lehec2013} verbatim until the step where we add an argument yielding the representation \eqref{eq:BSDE_charact}, starting~\eqref{eq:psi_process_ito} below. 
\begin{lemma}\label{lemma:generic_pathdep_control}
    Assume that $\E^\P[\exp({F(X)})] < \infty$ and $\E^\P[F^+(X)\exp({F(X)})] < \infty$. Then, there exists a weak control $\pi^* = (\Q^*, W^{\Q^*}, \alpha^*)$ with the controlled variable $X^{\alpha^*} = X$ which is a solution to the weak optimal control problem \eqref{eq:weak_primal_problem}, and
\begin{equation}\label{eq:optimality_eq}
        \sup_{\alpha \in \mathcal{A}}J(\alpha) = J(\alpha^*) = \log\E^\P[\exp({F(X)})].
    \end{equation}
    The value process $V^*$ is given by $V_t^* = \log\E^\P[\exp({F(X)})\,|\,\F_t]$ and admits the backward representation
\begin{equation}\label{eq:BSDE_charact}
        V_t^*  = F(X^{\alpha^*}) - \int_t^T\alpha_s^*\,dW_s^{\Q^*} - \dfrac12\int_t^T(\alpha_s^*)^2\,ds, \quad t \in [0, T].
    \end{equation}
    Moreover, the optimal control is closed-loop (adapted to the state): there exists a non-anticipative measurable functional $\Phi_t(\cdot)$ such that
    \begin{equation}\label{eq:feedback_functional}
        \alpha_t^* = \Phi_t(X^{\alpha^*}_s, \ s \in [0, t]).
    \end{equation}
\end{lemma}

\begin{proof}
    We first show that $\sup_{\alpha \in \mathcal{A}}J(\alpha) \leq \log Z$, where we denoted $Z:= \E^\P[e^{F(X)}]$. Consider an arbitrary admissible control $\pi = (\Q, W^{\Q}, \alpha)\in \mathcal A$ and define the measure $\Q'$ such that
    $$
    \dfrac{d\Q'}{d\Q} = \exp\left(-\int_0^T \alpha_s\, dW^{\Q}_s - \frac{1}{2}\int_0^T \alpha_s^2\, ds \right) =: N_T.
    $$
    Since the Doléans--Dade exponential $N$ on the right-hand side is a positive local martingale, it is a supermartingale, and we have $\E^{\Q}\left[\frac{d\Q'}{d\Q}\right] \leq 1$ and $\Q' \ll \Q$ (though it is not necessarily a probability measure). Localizing $N$ with $(\sigma_n)_{n \geq 1}$, we construct a sequence of probability measures $d\Q'_n = N_{T\land\sigma_n}\,d\Q$. It follows from Girsanov's theorem that 
    $$
    X_t^{\alpha, n} := W^{\Q}_t + \int_0^{t\land \sigma_n}\alpha_s\,ds \quad \text{is a $\Q'_n$-Brownian motion,}
    $$
    so that 
    $$
    Z = \E^{{\P}}[\exp({F(X)})] = \E^{\Q'_n}[\exp({F(X^{\alpha, n})})] = \E^{\Q}[N_{T \land \sigma_n}\exp({F(X^{\alpha, n})})].
    $$
    Since $X^{\alpha, n} = X^{\alpha}$ from some index $n \geq n^*(\omega)$ on, we have $N_{T \land \sigma_n}\exp({F(X^{\alpha, n})}) \to N_{T}\exp({F(X^{\alpha})})$ pointwise. By Fatou's lemma, we obtain
    \begin{equation}\label{eq:Z_inequality}
        \E^{\Q}[N_{T}\exp({F(X^{\alpha})})] \leq Z.
    \end{equation}
    Finally, it follows from \eqref{eq:Z_inequality} and Jensen's inequality that
    \begin{align}
        \log Z  &\geq \log{\E}^{\Q}\left[\exp\left(- \int_0^T\alpha_t\,dW^{\Q}_t - \dfrac12\int_0^T\alpha_t^2\,dt + F(X^\alpha) \right) \right] \geq {\E}^{\Q}\left[- \dfrac12\int_0^T\alpha_t^2\,dt + F(X^\alpha)\right] = J(\alpha),
    \end{align}
    where we used that $\E^{\Q}[\int_0^T\alpha_t^2\,dt] < \infty$ by the definition of admissibility.
    This proves $\sup_{\alpha \in \mathcal{A}}J(\alpha) \leq \log Z$.

    We now construct $\pi^*\in\mathcal{A}$ for which the optimal value $\log Z$ is attained.    We define a measure $\Q^*$ associated to the Radon--Nikodym derivative
    $$
    \dfrac{d\Q^*}{d\P} := \dfrac{\exp({F(X)})}{Z}, \qquad M_t := \E^\P\bigg[\dfrac{d\Q^*}{d\P}\,\bigg|\,\F_t\bigg].
    $$
    Note that $M$ is an a.s.~positive $\P$-martingale with $M_0 = 1$ and $M_T = \frac{e^{F(X)}}{Z}$. By the martingale representation theorem, there exists an $\mathbb{F}$-predictable process $h$ such that
    $
    \P\left(\int_0^Th_t^2\,dt < \infty \right) = 1,
    $
    and
    $$
    M_t = 1 + \int_0^t h_s\,dX_s = 1 + \int_0^t \alpha_s^* M_s\,dX_s, \quad t \in [0, T],
    $$
    where we defined
    \begin{equation}\label{eq:alpha_star_def}
        \alpha^*_t := \dfrac{h_t}{M_t}\indic{M_t > 0}.
    \end{equation}
    The process $\alpha^*$ is adapted and verifies 
    $
    \P\left(\int_0^T(\alpha_t^*)^2\,dt < \infty \right) = 1,
    $
    see the proof of \cite*[Theorem 2]{Lehec2013} for details. Therefore, $M = \exp(\int_0^{\cdot} \alpha_s^* dX_s - \frac12 \int_0^{\cdot} (\alpha_s^*)^2 ds)\,$. Since $M$ is also a true martingale, Girsanov's theorem applies: the process
    $$
    W^{\Q^*}_t := X_t - \int_0^t\alpha_s^*\,ds \quad \text{is a $\Q^*$-Brownian motion.}
    $$
    Hence, the optimal control candidate is $\pi^* = (\Q^*, W^{\Q^*}, \alpha^*)$ with the corresponding controlled variable $X^{\alpha^*} = W^{\Q^*} + \int_0^\cdot\alpha_s\,ds = X$ coinciding with the coordinate process.
    
    We define the process $V_t := \log\E^\P[\exp({F(X)})\,|\,\F_t]$. We have
    \begin{equation}\label{eq:psi_process_ito}
    \begin{aligned}
        V_t = \log(Z M_t) = \log(Z) + \int_0^t \alpha^*_s dW_s^{\Q^*} + \frac{1}{2} \int_0^t (\alpha_s^*)^2 ds, \quad t \in [0, T],
    \end{aligned}    
    \end{equation}
    with $V_T = F(X) = F(X^{\alpha^*})\,$. This yields \eqref{eq:BSDE_charact}.
    Note that since $M$ is a positive $\P$-martingale, $M^{-1}$ is a $\Q^*$-martingale, so that $V_t = \log Z-\log M_t^{-1}$ is a $\Q^*$-submartingale.
    Localizing $\alpha^*$ with $(\tau_n)_{n \geq 1}$ and taking the expectation $\E^{\Q^*}$, we obtain
    \begin{equation}\label{eq:psi_localization}
        \E^{\Q^*}\left[\dfrac12\int_0^{T\land\tau_n}(\alpha_s^*)^2\,ds\right] = \E^{\Q^*}\left[ V_{T \land \tau_n} \right] -\log Z = \E^{\Q^*}\left[ V^+_{T \land \tau_n} \right] - \E^{\Q^*}\left[V^-_{T \land \tau_n} \right] - \log Z.
    \end{equation}
    The left-hand side converges to $\E^{\Q^*}\left[\frac12\int_0^{T}(\alpha_s^*)^2\,ds\right]$ by the monotone convergence theorem. For the negative part of $V_{T \land \tau_n}$, we have
    $$
    \E^{\Q^*}\left[ V^-_{T \land \tau_n} \right] = \E^{\P}\left[M_{{T \land \tau_n}} \log^- (ZM_{T \land \tau_n}) \right].
    $$
    The integrand on the right-hand side is a.s. bounded from above uniformly in $n \geq 1$ by
    $$
    M_{{T \land \tau_n}} \log^- (ZM_{T \land \tau_n}) = -M_{{T \land \tau_n}} \log(ZM_{T \land \tau_n})\indic{ZM_{T \land \tau_n} < 1} \leq Z^{-1}\exp(-1)
    $$
    since $-x\log(x) \leq \exp(-1)$ on $(0, 1]$. Hence, by the dominated convergence theorem, $\E^{\Q^*}\left[ V_{T \land \tau_n}^- \right]$ converges to $\E^{\P}[M_T \log^-(ZM_T)] = \E^{\Q^*}\left[ V_{T}^- \right] < \infty$ as $n \to +\infty$. For the positive part, we use the submartingale property of $V$: $V_{T\land\tau_n} \leq \E^{\Q^*}\left[ V_{T} \,|\, \F_{T\land\tau_n} \right]$, so that 
    $$
    V^+_{T\land\tau_n} \leq \E^{\Q^*}\left[ V_{T} \,|\, \F_{T\land\tau_n} \right]^+ \leq \E^{\Q^*}\left[ V^+_{T} \,|\, \F_{T\land\tau_n} \right] = \E^{\Q^*}\left[F^+(X) \,|\, \F_{T\land\tau_n} \right].
    $$
    As $\E^{\Q^*}[F^+(X)] = \E^{\P}[F^+(X)\exp({F(X)})] < \infty$ by assumption, the right-hand side is $\Q^*$-uniformly integrable and hence $(V^+_{T\land\tau_n})_{n \geq 1}$ is $\Q^*$-uniformly integrable as well. This implies $\E^{\Q^*}\left[ V^+_{T \land \tau_n} \right] \to \E^{\Q^*}\left[ V^+_{T} \right] < \infty$ as $n \to +\infty$. Passing to the limit in \eqref{eq:psi_localization}, we conclude
    \begin{equation}\label{eq:relative_entropy}
         \E^{\Q^*}\left[\dfrac12\int_0^{T}(\alpha_s^*)^2\,ds\right] = \E^{\Q^*}\left[ V_{T} \right] -\log Z = \E^{\Q^*}\left[ F(X) \right] - \log Z < \infty.
    \end{equation}
    In particular, this implies 
    \begin{equation}\label{eq:optimality_eq_proof}
        J(\alpha^*) = \E^{\Q^*}\left[-\frac12\int_0^{T}(\alpha_s^*)^2\,ds + F(X)\right] = \log Z,
    \end{equation}
    which yields the second equality in \eqref{eq:optimality_eq}.

    We now identify the value process. Repeating the argument above with conditional expectations in place of expectations and the conditional Girsanov transform in place of Girsanov's theorem, one obtains
    $$
    {\E}^{\Q}\left[- \dfrac12\int_t^T\alpha_s^2\,ds + F(X^\alpha) \,\Bigg|\, {\F}_t\right] \leq \log\E^{\P}\left[\exp({F(X)}) \,\big|\, \F_t\right], \quad a.s.,
    $$
    for all $\pi = (\Q, W_s^{\Q}, \alpha)\in \mathcal A_{t, T}$, and
    $$
    {\E}^{\Q^*}\left[- \dfrac12\int_t^T(\alpha_s^*)^2\,ds + F(X) \,\Bigg|\, \F_t\right] = \log\E^{\P}\left[\exp({F(X)})\,\big|\, \F_t\right], \quad a.s.
    $$
    By the definition of the essential supremum, $V_t^* = \log\E^{\P}\left[\exp({F(X)})\,\big|\,\F_t\right] = V_t$. The backward representation \eqref{eq:BSDE_charact} follows from \eqref{eq:psi_process_ito}.

    Finally, \eqref{eq:feedback_functional} follows from the fact that $\alpha^*$ is $\mathbb{F}$-predictable by construction and that $\mathbb{F}$ is generated by the canonical process $X = X^{\alpha^*}$.
\end{proof}

\begin{sqremark}\label{cor:entropic_duality}
    The measure $\Q^*$ constructed in the proof of Lemma~\ref{lemma:generic_pathdep_control} is the unique maximizer in the relative entropy duality
    \begin{equation}\label{eq:entropy_duality}
        \sup_{\Q\ll\P}\left\{\E^{\Q}[F(X)] - H(\Q \,|\, \P)\right\} = \log \E^\P[\exp({F(X)})].
    \end{equation}
    Moreover, Theorem \ref{thm:oc_sig_repr} makes the measure $\Q^*$ simulable in explicit feedback form via the Riccati equation.
    Indeed, by construction, $\Q^* \sim \P$, and the relative entropy
    $$
    H(\Q^* \,|\, \P)  = \E^{\Q^*}\left[\log\left(\frac{d\Q^*}{d\P}\right)\right] = \E^{\Q^*}\left[ V^*_{T} \right] - \log \E^\P[\exp({F(X)})] = \E^{\Q^*}\left[\dfrac12\int_0^{T}(\alpha_s^*)^2\,ds\right] < \infty,
    $$
    by \eqref{eq:relative_entropy}, so that 
    $$
    \E^{\Q^*}[F(X)] - H(\Q^* \,|\, \P) = \E^{\Q^*}\left[-\dfrac12\int_0^{T}(\alpha_s^*)^2\,ds+ F(X)\right] = \log \E^\P[\exp({F(X)})],
    $$
    by \eqref{eq:optimality_eq}.
    The uniqueness of the maximizer follows from the strict concavity of the objective functional in \eqref{eq:entropy_duality}.
\end{sqremark}

\begin{proof}[Proof of Theorem~\ref{thm:oc_sig_repr}]
    We recall that $X^{\alpha^*} = X$ is a $\P$-Brownian motion. Firstly, we easily check that $F(X) = \bracket{\bp}{\sigXhat[T]}$ satisfies the assumptions of Lemma \ref{lemma:generic_pathdep_control}. Indeed, since $X$ is Brownian motion under the Wiener measure $\mathbb P\,$, we know from \cite*[Proposition 4.3]{riccatifourierlaplace} that $\bracket{\bp}{\sigXhat[T]} \leq C_T$ a.s. for some constant $C_T \geq 0\,$. This proves that $\E^{\P}[\exp({F(X)})]$ and $\E^{\P}[F^+(X) \exp({F(X)})]$ are finite.

    For readability, we denote ${}_{\sigXastarhat[\sigma]}|\bp$ by $\bp_{\sigma}^*\,$, and $\bpsi_t^{\sigXastarhat[\sigma]}$ by $\bpsi^*_{\sigma,t}\,$.
    The fact that $\tau_{\sigma} > \sigma$ a.s. comes from the first point in Theorem \ref{thm:results_papier1}, which gives $r = r(\bp_{\sigma}^*, T, \varepsilon) > 0$ such that $\| \shuexp{\overline{\bpsi^*_{\sigma,t}}} \|_{\sigX[]} \leq 2 - \varepsilon$ for any $\sigX[] \in G_{\bp^*_{\sigma}, r}\,$, i.e., satisfying $M_{\bp_{\sigma}^*}(\sigX[]) \leq r\,$. Hence, defining $\tilde{\tau}_{\sigma} := \inf\{t \geq \sigma\,:\, M_{\bp^*_{\sigma}}(\sigXastarhat[\sigma,t]) \geq r \}\,$, we have $\tau_{\sigma} \geq \tilde{\tau}_{\sigma}$ almost surely. Since $M_{\bp_{\sigma}^*}(\sigXastarhat[\sigma,\sigma]) = M_{\bp_{\sigma}^*}(\emptyword) = 0\,$, and $t \mapsto M_{\bp^*_{\sigma}}(\sigXastarhat[\sigma,t])$ is continuous, we have $\tilde{\tau}_{\sigma}> \sigma$ almost surely, thereby proving $\tau_{\sigma} > \sigma\,$.

    We now aim to identify the value process and the optimal control on $\llbracket \sigma\,, \tau_{\sigma} \rrbracket\,$. By Theorem \ref{thm:results_papier1} and the definition of $\tau_{\sigma}\,$, the value process can be written as 
    \begin{align*}
    V^*_t \,\indic{t \in \llbracket \sigma, \tau_{\sigma}\rrbracket} &= \log \E^{\P}[\exp(\bracket{\bp}{\sigXastarhat[T]})\,|\,\mathcal F_t]\,\indic{t \in \llbracket \sigma, \tau_{\sigma}\rrbracket} \\
    &= \log \E^{\P}[\exp(\bracket{\bp}{\sigXastarhat[\sigma] \otimes \sigXastarhat[\sigma,t] \otimes \sigXastarhat[t,T]})\,|\,\mathcal F_t]\,\indic{t \in \llbracket \sigma, \tau_{\sigma}\rrbracket} \\ 
    &= \log \E^{\P}[\exp(\bracket{\bp_{\sigma}^*}{\sigXastarhat[\sigma,t] \otimes \sigXastarhat[t,T]})\,|\,\mathcal F_t]\,\indic{t \in \llbracket \sigma, \tau_{\sigma}\rrbracket} \\
    &= \bracket{\bpsi^*_{\sigma,t}}{\sigXastarhat[\sigma,t]}\,\indic{t \in \llbracket \sigma, \tau_{\sigma}\rrbracket}\,, \quad t \in [0\,,T]\,,
    \end{align*}
    where we used Chen's identity \eqref{eq:chen}, the recentering \eqref{eq:recentering_bracket}, as well as the fact that $\bp^*_{\sigma} \in \mathcal B$ almost surely thanks to Proposition~\ref{prop:left-proj-compatibility}. This identifies the value process. The bounds \eqref{eq:bound_psi_proj_v} show that
    $$
    \int_{\sigma}^{t \wedge \tau_{\sigma}} \|\bpsi^*_{\sigma,s}|_{\wv}\|_{\sigXastarhat[\sigma,s]} ds \leq \int_{\sigma}^{t \wedge \tau_{\sigma}} f_{T, \varepsilon, \wv}(\bp^*_{\sigma}, s, M_{\bp^*_{\sigma}}(\sigXastarhat[\sigma,s]))ds\,, \quad t \in [0, T]\,,
    $$
    which is finite by continuity of $f$ in its second and third arguments, and by continuity of $s \mapsto M_{\bp_{\sigma}^*}(\sigXastarhat[\sigma,s])\,$. The same can be done for the time derivative $\dot{\bpsi}^*_{\sigma, \cdot}$ which can be expressed using the shifts of $\dot{\bpsi}^*_{\sigma, \cdot}$ by virtue of the Riccati equation $\dot{\bpsi}^*_{\sigma, \cdot} = - \bpsi^*_{\sigma, \cdot}|_{\word 0} - \frac{1}{2} \bpsi^*_{\sigma,\cdot}|_{\word{11}} - \frac{1}{2} (\bpsi^*_{\sigma,\cdot}|_{\word 1})\shupow{2}$, using the shuffle compatibility \eqref{eq:shuffle_compatibility} of the seminorm $\|\cdot \|_{\sigX[]}\,$. We can therefore apply It\^o's formula to the process $(V^*_t)_{t \in \llbracket\sigma, \tau_{\sigma}\rrbracket} = (\bracket{\bpsi^*_{\sigma,t}}{\sigXastarhat[\sigma,t]})_{t \in \llbracket \sigma, \tau_{\sigma}\rrbracket}\,$, see \cite*[Theorem 3.5]{riccatifourierlaplace}, yielding 
    \begin{align*}
        V^*_{t} &= V^*_{\sigma} + \int_{ \sigma}^{t \wedge \tau_{\sigma}} \bracket{\dot{\bpsi^*}_{\sigma,s} + \bpsi^*_{\sigma,s}|_{\word 0} + \frac{1}{2} \bpsi^*_{\sigma,s}|_{\word{11}}}{\sigXastarhat[\sigma,s]}\, ds + \int_{\sigma}^{t \wedge \tau_{\sigma}} \bracket{\bpsi^*_{\sigma,s}|_{\word 1}}{\sigXastarhat[\sigma,s]} \, dX^{\alpha^*}_s \\
        &= V^*_{\sigma} - \frac{1}{2}\int_{\sigma}^{t \wedge \tau_{\sigma}} \bracket{\bpsi^*_{\sigma,s}|_{\word 1}}{\sigXastarhat[\sigma,s]}^2\, ds + \int_{\sigma}^{t \wedge \tau_{\sigma}} \bracket{\bpsi^*_{\sigma,s}|_{\word 1}}{\sigXastarhat[\sigma,s]}\, dX^{\alpha^*}_s\,, \quad t \in \llbracket \sigma\,, T\rrbracket\,,
    \end{align*}
    where we used the Riccati equation satisfied by $\bpsi^*_{\sigma, \cdot}$ to simplify the drift. By \eqref{eq:BSDE_charact}, 
    \begin{align*}
        dV^*_t = \frac{1}{2} (\alpha^*_t)^2 dt + \alpha^*_t dW_t^{\Q^*} = - \frac{1}{2} (\alpha^*_t)^2 dt + \alpha^*_t dX^{\alpha^*}_t \,.
    \end{align*}
    Therefore,  the processes $(\alpha^*_t)_{t \in \llbracket \sigma , \tau_{\sigma} \rrbracket}$ and $(\bracket{\bpsi^*_{\sigma, t}|_{\word 1}}{\sigXastarhat[\sigma,t]})_{t \in \llbracket \sigma, \tau_{\sigma} \rrbracket}$ coincide $dt \otimes d\P$-a.e., which concludes the proof.
\end{proof}

\section{Applications}\label{section:application}
{In this section, we consider several applications covered by Theorem~\ref{thm:oc_sig_repr} and discuss some numerical aspects of its implementation. 

\subsection{Tracking linear functionals of the signature}\label{subsection:tracking}
    In the first application, we consider $K \in \N$ processes $Y^k, k = 1, \ldots, K,$ which depend on the path of the controlled variable and are represented by the linear functionals 
    $$
    Y_t^k = Y_t^k((X_s^\alpha)_{s \in [0, t]}) = \bracket{\bell_k}{\sigXahat}, \quad t \in [0, T], \quad k = 1, \ldots, K,
    $$
    for some $\bell_k \in T((\R^2))$ with $\deg(\bell_k) = N_k \in \N$ and $\bell_k^{\word{1}\conpow{N_k}} \neq 0$. The control problem consists of maintaining each process $Y^k$ close to a target value $\mu_k \in \R$ with an  objective functional of the form
    \begin{equation}\label{eq:tracking_example}
        J(\alpha) = \E^\Q\left[-\dfrac12\int_0^T\alpha_t^2\,dt - \sum_{k=1}^K\int_0^T(Y_t^k - \mu_k)^{2m_k}\,dt + \bracket{{\bm g}}{\sigXahat[T]}\right],
    \end{equation}
    for some $m_k \in \N,\ k = 1, \ldots, K$ and the terminal reward coefficient ${\bm g} \in T((\R^2))$ such that 
\begin{equation}\label{eq:coef_condition_tracking}
        \left(-\sum_{k=1}^K(\bell_k - \mu_k\emptyword)\shupow{2m_k}\right)\word{0} + {\bm g} \in \mathcal{B}.
    \end{equation}
    
 For such coefficients, Theorem~\ref{thm:oc_sig_repr} and Corollary~\ref{C:running} can be applied to solve the maximization problem of \eqref{eq:tracking_example}.  This is the case for instance  for ${\bm g} = 0$.

    For the numerical illustration, we consider the  particular case of \eqref{eq:tracking_example} given by \begin{equation}\label{eq:objective_numerical_illustr}
        J(\alpha) = \E^\Q\left[-\dfrac12\int_0^T\alpha_t^2\,dt  -w_{{\bm f}}\int_0^T(X_t^\alpha - \mu)^{4}\,dt - w_{{\bm g}}\sum_{|\word{v}| \leq 2}(\sigXhat[T]^{\alpha, \word{v}} - \sigYhat[T]^{\word{v}})^2\right],
    \end{equation}
    where $\mu \in \R$, $w_{\bm f} > 0,\ w_{\bm g} \geq 0$, and $\sigYhat[]$ denotes the time-augmented signature of a deterministic function $t \mapsto Y(t)$. This corresponds to \eqref{eq:tracking_example}  for the choice  $K=1$, $\bell_1=\word 1$, $m_1 = 2$ and $\bm g$ as in \eqref{eq:bfbg} below.
     A straightforward verification of the definition shows that the coefficient corresponding to this objective functional belongs to the class $\mathcal{B}$.
    Using the decomposition into running and terminal reward from Corollary~\ref{C:running}, we write
    \begin{align}\label{eq:bfbg}
    {\bm f} = -w_{\bm f}(\word{1} - \mu\emptyword)\shupow{4}, \qquad {\bm g} = -w_{\bm g}\sum_{|\word{v}| \leq 2}(\word{v} - \sigYhat[T]^{\word{v}} \emptyword)\shupow{2},
    \end{align}
    and solve the Riccati equation \eqref{eq:Riccati_recentered_source} with source ${\bm f}$ to obtain $\bchi^{\sigX[]}$. The value process and optimal control are then given by \eqref{eq:value_control_local} with $\bpsi^{\sigX[]} = \bchi^{\sigX[]} + ~_{\sigX[]}|({\bm f}\word{0})$. By Corollary~\ref{C:running}, in order to ensure the convergence of the representation \eqref{eq:value_control_local} after recentering at some time $\tau$, one must verify that $\|{\exp^{\shuprod}{(\overline{\bchi}_t^{\sigXastarhat[\tau]})}}\|_{\sigXastarhat[\tau, t]} <  (2 - \varepsilon).$ 
    Thus, the dynamic recentering procedure can be summarized as follows:
    \begin{enumerate}
        \item At $t=0$, set $\tau_0 = 0$.
        \item For $n \geq 0$, at the time $\tau_n$ solve numerically the recentered Riccati equation \eqref{eq:Riccati_recentered_source} with $\sigX[] = \sigXastarhat[\tau_n]$ to obtain the coefficients 
        $(\bchi_t^{\sigXastarhat[\tau_n]})_{ t \in \llbracket \tau_n, T\rrbracket}$.
        \item Use the value function and the optimal control representation
        $$
        V_t^* = \bracket{\bchi^{\sigXastarhat[\tau_n]}_t}{\sigXastarhat[\tau_n, t]} + \int_0^t\bracket{{\bm f}}{\sigXastarhat[s]}\,ds, \qquad \alpha^*_t =  \bracket{\bchi_t^{\sigXastarhat[\tau_n]}\proj{1}}{\sigXastarhat[\tau_n,t]}, \qquad t \in \llbracket \tau_n, \tau_{n+1}\rrbracket,
        $$
        up to the next recentering time 
        \begin{equation}\label{eq:recentering_rule}
            \tau_{n+1} = \inf\{t > \tau_n\colon\ \|{\exp^{\shuprod}{(\overline{\bchi}_t^{\sigXastarhat[\tau_n]})}}\|_{\sigXastarhat[\tau_n, t]}  \geq  2 - \varepsilon \}.
        \end{equation}
    \end{enumerate}

We apply this algorithm to the control problem corresponding to the objective functional \eqref{eq:objective_numerical_illustr} with the parameters 
$$
\mu = 1.4, \qquad w_{{\bm f}} = w_{{\bm g}} =  5,\qquad  Y(t) = 1.4 \cdot \left(\frac{2t}{T} \land 1\right),
$$
over the time horizon $T = 0.4$. We take $\varepsilon = 0.01$ in \eqref{eq:recentering_rule}.

The Riccati equation \eqref{eq:Riccati_recentered_source} truncated at the signature level $N_{\mathrm{trunc}} = 10$ is then solved on the uniform grid with $n_{\mathrm{t}} = 300$ time steps using the predictor-corrector scheme for ODEs. Figure~\ref{fig:control_constant_mu} shows two solutions of the optimal control problem obtained via the Riccati equation-based algorithm described above (solid lines) and the Monte Carlo benchmark \eqref{eq:value_fct_control_mc} with $5\cdot 10^5$ samples of the Brownian signature. To highlight the importance of the recentering procedure, we also provide the trajectories (dashed) corresponding to the control obtained via solving the Riccati equation once at $t = 0$ and using the solution $\bchi$ to construct the control and the value process without recentering as
\begin{equation}\label{eq:Riccati_no_recentering}
    V_t^* \approx \bracket{\bchi_t}{\sigXastarhat[t]} + \int_0^t\bracket{{\bm f}}{\sigXastarhat[s]}\,ds, \qquad \alpha^*_t \approx \bracket{\bchi_t\proj{1}}{\sigXastarhat[t]}, \qquad t \in [0, T].
\end{equation}

\begin{figure}[H]
    \centering
    \includegraphics[width=\linewidth]{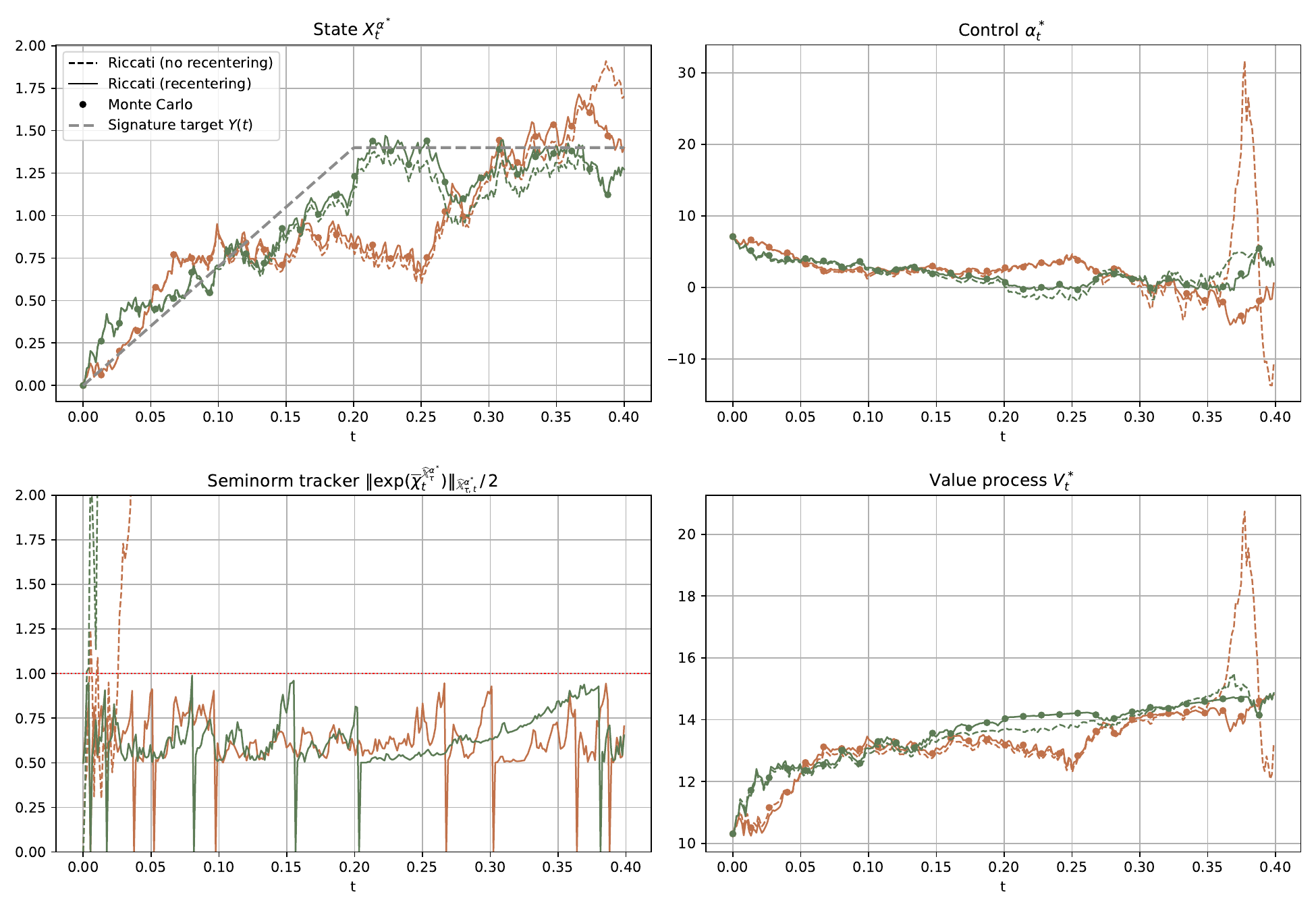}
    \label{fig:psi_series}
    \caption{Two realizations (green and brown) of the solution to the stochastic optimal control problem \eqref{eq:objective_numerical_illustr} obtained via solving the Riccati equation with / without recentering (solid / dashed) and via the Monte Carlo estimation of \eqref{eq:value_fct_control_mc} (dotted). The upper left plot shows the state variable $X^\alpha$ constructed using the control given by each of the three methods and the signature target function $Y(t)$ (dashed gray), the upper and lower right plots show the trajectories of the optimal control and the value process evaluated by each method, and the lower left plot shows the trajectories of the seminorm tracker.}
    \label{fig:control_constant_mu}
\end{figure}

Providing accurate results for moderate values of $t$, the uncentered Riccati expansion \eqref{eq:Riccati_no_recentering} demonstrates a significant deviation from both the optimal control and the optimal trajectory of the controlled variable for larger values of $t$ when the signature $\sigXastarhat[t]$ moves far away from the expansion origin $\emptyword$ in the tensor algebra, as indicated by the seminorm tracker
$
t \mapsto \frac{1}{2}\|{\exp^{\shuprod}{(\overline{\bchi}_t^{\sigXastarhat[\tau]})}}\|_{\sigXastarhat[\tau, t]}
$
on the lower-left plot. In this expression, $\tau$ denotes the latest recentering time and the new recentering should be performed when the seminorm tracker attains $(1 - \frac{\varepsilon}{2})$, according to \eqref{eq:recentering_rule}.

The explosive behavior of this deviation provides indirect numerical evidence that the signature expansion of the value process (and, of course, of the optimal control as well) is local and the corresponding path-dependent map is only analytic, but not entire, as stated in Section~\ref{section:riccati_prelim}.

We remark that the development of faster and more efficient schemes adapted to the tensor algebra structure and able to handle the shuffle non-linearity and stiffness of the equation remains a challenging task, see, for instance, \cite*{abijaber2024fourier} for the Markovian case.

We also demonstrate that the case of the time-dependent coefficient of the running reward discussed in Remark~\ref{rmk:timedep_running} can be treated numerically in a similar manner. Namely, we consider the same stochastic optimal control problem except for the time-dependent running target 
 $\mu(t) = 2\sqrt{t}$, hence yielding the time-dependent source coefficient $({\bm f}_t)_{t \in [0, T]}$. This is solved using the Riccati equation with the time-dependent source as discussed in Remark~\ref{rmk:timedep_running}. The numerical results are presented in Figure~\ref{fig:control_timedep_mu}.

 \begin{figure}[H]
    \centering
    \includegraphics[width=\linewidth]{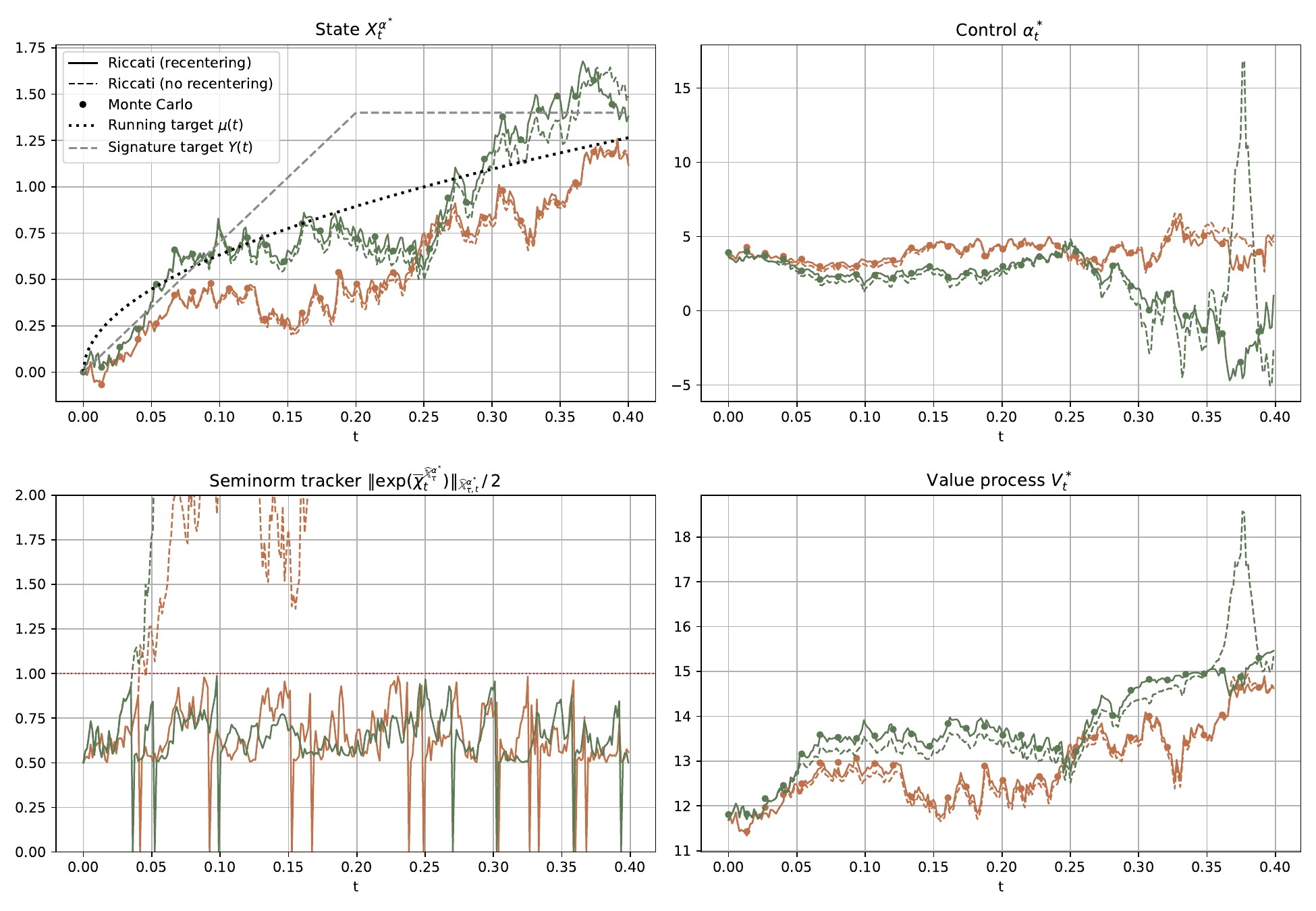}
    \label{fig:psi_series}
    \caption{Two realizations (green and brown) of the solution to the stochastic optimal control problem \eqref{eq:objective_numerical_illustr} with time-dependent $\mu(t)$ obtained via solving the Riccati equation with / without recentering (solid / dashed) and via the Monte Carlo estimation of \eqref{eq:value_fct_control_mc} (dotted). The upper left plot shows the state variable $X^\alpha$ constructed using the control given by each of the three methods, the signature target function $Y(t)$ (dashed gray), and the running target $\mu(t)$ (dotted black). The upper and lower right plots show the trajectories of the optimal control and the value process evaluated by each method, and the lower left plot shows the trajectories of the seminorm tracker.}
    \label{fig:control_timedep_mu}
\end{figure}

 \begin{sqremark}
    Instead of tracking the seminorm condition as suggested by \eqref{eq:recentering_rule}, one may decide to recenter at each time step $t_k$, solving a new Riccati equation to obtain the coefficients $\bchi_{t_k}^{\sigXastarhat[t_k]}$. The value function and the optimal control representation in this case contain only one term and are given simply by
    $$
    V_{t_k}^* = \bracket{\bchi_{t_k}^{\sigXastarhat[t_k]}}{\emptyword} + \int_0^t\bracket{{\bm f}}{\sigXastarhat[s]}\,ds, \qquad \alpha_{t_k}^* = \bracket{\bchi_{t_k}^{\sigXastarhat[t_k]}\proj{1}}{\emptyword}.
    $$
    The price to pay is of course the computational cost associated with solving a new Riccati equation at each discretization step instead of solving it on a generally much sparser set of stopping times \eqref{eq:recentering_rule}. 
\end{sqremark}
}

\subsection{Signature lifts of  Volterra control problems}\label{sect:volterra}

We show how our framework nests certain non-Markovian Volterra control problems. {\Dim For $\alpha \in \mathcal{A}$}, let $Y^\alpha$ be a controlled Volterra process defined by
\begin{align}
Y_t^{\alpha}
&= y_0 + \int_0^t K(t,s)\big(\alpha_s\,ds + {dW_s^{\Q}}\big) = y_0 + \int_0^t K(t,s)\, dX_s^{\alpha},
\end{align}
where {$X^\alpha$ is given by \eqref{eq:controlled_var_def_weak}}, $y_0 \in \mathbb R$ and $K:[0,T]^2 \to \mathbb{R}$ is a measurable kernel satisfying
$$\sup_{t \le T} \int_0^T K(t,s)^2 ds < \infty.$$ 
We consider a reward to be maximized given by  
\begin{align}\label{eq:costvolterra}
J_{\text{Volterra}} (\alpha) =\mathbb{E}\!\left[\int_0^T \left(-\frac{\alpha_s^2}{2} + \nu (Y_s^\alpha)^{2m} + \eta (X^{\alpha}_s)^{2n} \right)ds \right],
\end{align}
for some $\eta,\nu  \leq 0$ and $n,m \in \mathbb N$.  For $m,n \in \{0,1\}$, the problem reduces to a linear–quadratic Volterra control problem. Variants of such problems in the linear-quadratic setting have been studied in various applications, notably in optimal execution with transient market impact with a Volterra propagator; see for instance \cite*{abi2025optimal}.  For $m,n  \ge 2$, no general explicit solution is known, and classical dynamic programming approaches fail due to the lack of a finite-dimensional Markovian state.

We now show how this class of Volterra control problems can be embedded into our framework by writing $J_{\text{Volterra}}$ in the form \eqref{eq:weak_primal_problem}. The key observation is that, for a broad class of kernels, the Volterra process $Y^\alpha$ admits a representation as a linear functional of the time-augmented signature of the driving process $X^\alpha$, namely
\begin{align}\label{eq:YtoSig}
Y_t^\alpha = \langle \boldsymbol{k}_t,\widehat{\mathbb X}_t^\alpha\rangle, 
\end{align}
for some coefficient sequence $\boldsymbol{k}_t$, which may be either time-dependent or time-independent and may contain finitely or infinitely many nonzero terms. Such representations have recently been derived in \cite*{abi2024path} for broad classes of kernels. 

The following proposition provides two important examples.

\begin{proposition}\label{P:kernels}
For kernels  of the form
$K(t,s)=\sum_{i=0}^{N} k_i(t)\frac{s^i}{i!},$
for some coefficients $(k_i(t))_{i\ge0}$,  the representation \eqref{eq:YtoSig} holds with time-dependent coefficients
$$\boldsymbol{k}_t = y_0\emptyword + \sum_{i=0}^{N} k_i(t)\word{0}\conpow{i}\word{1}.$$
For kernels  of the form
$K(t,s)=\sum_{i=1}^N c_i e^{-\lambda_i(t-s)},$
for some $c_i,\lambda_i\in\mathbb R$, the representation \eqref{eq:YtoSig} holds with time-independent coefficients
$$
\boldsymbol{k} = y_0\emptyword+\sum_{i=1}^N c_i\big(\word{1}\otimes\shuexp{-\lambda_i \word 0}\big).
$$
\end{proposition}

\begin{proof}
We only sketch the proof of the first case. Note that for $i \geq 0$, $$\frac{s^i}{i!}= \int_{0<s_1<\ldots <s_i} ds_1 \ldots ds_i  = \langle \word{0}^{\otimes i} , \widehat{\mathbb X}^{\alpha}_s\rangle,   $$ so that 
$ \int_0^t \frac{s^i}{i!} dX_s^\alpha = \int_0^t \langle \word{0}\conpow{i}, \widehat{\mathbb X}^{\alpha}_s\rangle dX_s^\alpha  = \langle \word{0}\conpow{i} \word{1} , \widehat{\mathbb X}^{\alpha}_t\rangle.$ Substituting the power series expansion of $K$ into the definition of $Y^\alpha$ and exchanging summation and integration yields  $$
Y_t^\alpha
=y_0+\sum_{i=0}^{N}k_i(t)\int_0^t \frac{s^i}{i!} dX_s^\alpha = y_0 +  \sum_{i=0}^{N}k_i(t) \langle \word{0}\conpow{i} \word{1} , \widehat{\mathbb X}^{\alpha}_t\rangle $$
which proves the first point. The second case follows directly from the signature representation of Ornstein Uhlenbeck-type processes as done in  \cite*[Example 4.3]{abi2024path}.
\end{proof}

\begin{remark}
The two examples in Proposition~\ref{P:kernels} cover broad classes of kernels and also provide natural approximation schemes. For instance, fractional kernels of the form $K(t,s)=(c+t-s)^{\beta-1}\mathbf{1}_{{s<t}}$, with $c\geq 0$ and $\beta>0$, admit an infinite power series expansion of the form
$$
K(t,s)=\sum_{i=0}^{\infty}(-1)^i\frac{(\beta - 1)(\beta -2)\ldots(\beta - i)}{i!}(c+t)^{\beta-1-i}{s^i},
$$

and truncating at finite order yields a natural approximation. Likewise, finite sums of exponentials form a flexible class that can approximate a wide range of kernels. We refer to \cite*{abi2024path} for further examples and discussion.
\end{remark}

As a consequence, by Proposition~\ref{P:kernels} and the shuffle property, polynomial rewards in $Y^\alpha$ can be rewritten as polynomial functionals of the signature of $X^\alpha$, so that the Volterra control problem \eqref{eq:costvolterra} falls within our framework:
\begin{align}
J_{\text{Volterra}}(\alpha)
=\mathbb E\left[\int_0^T \left(-\frac{\alpha_s^2}{2}+\langle {\bm f}_s,\widehat{\mathbb X}_s^\alpha\rangle\right)ds\right],
\end{align}
with
\begin{align}
{\bm f}_s = {\nu\boldsymbol{k}_s^{\shuprod 2m}+ \eta(2n)!\word{1}^{\otimes 2n}},
\end{align}

where $\boldsymbol{k}_s$ is given as in Proposition~\ref{P:kernels}. In particular, whenever the orders $m,n$ and coefficients are chosen so that our main Theorem~\ref{thm:oc_sig_repr} applies, this yields a tractable representation and implementation of the optimal control for a nontrivial class of genuinely nonlinear and non-Markovian Volterra control problems. For instance, this is the case for truncated time-independent kernels $\bm k$ with $m < n$ and $\eta < 0$, while the case of time-dependent coefficients is covered by the conjectured extension of Remark~\ref{rmk:timedep_running}. While our current theory is formulated under specific structural assumptions within the class $\mathcal B$, the above representation remains valid more generally and suggests that the scope of the method extends beyond the present setting.

\bibliographystyle{plainnat}
\bibliography{refs.bib}

\end{document}